\documentclass[11pt,reqno]{amsart}
\usepackage[ocgcolorlinks,unicode,bookmarks]{hyperref}
\usepackage[dvipsnames]{xcolor}
\hypersetup{colorlinks=true,citecolor=NavyBlue,linkcolor=BrickRed,urlcolor=Green}
\hypersetup{pdftitle={Rank 2 Affine Invariant Subvarieties in H(6)},
  pdfauthor={Pramana Saldin and Ruocheng Yang}}
\usepackage[nameinlink]{cleveref}
\usepackage{anysize} \marginsize{1.3in}{1.3in}{1in}{1in}
\usepackage{amsfonts}
\usepackage[figuresleft]{rotating}

\usepackage{graphicx}

\usepackage{booktabs}

\let\oldtocsection=\tocsection

\let\oldtocsubsection=\tocsubsection

\let\oldtocsubsubsection=\tocsubsubsection

\renewcommand{\tocsection}[2]{\bfseries\hspace{0em}\oldtocsection{#1}{#2}} %
\renewcommand{\tocsubsection}[2]{\hspace{1em}\oldtocsubsection{#1}{#2}}
\renewcommand{\tocsubsubsection}[2]{\itshape\hspace{2em}\oldtocsubsubsection{#1}{#2}}

\usepackage{tikz-cd}
\usepackage{float}
\usetikzlibrary{calc, positioning, decorations.markings,arrows,tqft,shapes.geometric}

\usepackage{comment}
\usepackage[all,cmtip]{xy}
\usepackage{amsmath}
\usepackage{enumitem}
\usepackage{amsthm}
\usepackage{amssymb}
\usepackage{mathrsfs}
\usepackage{enumitem}

\usepackage{tikz}
\tikzset{
  labl/.style={anchor=south, rotate=90, inner sep=.5mm}
}
\colorlet{lightgreen}{green!10!white}

\usepackage[normalem]{ulem}

\newtheorem{thm}{Theorem}[section]

\newtheorem{prop}[thm]{Proposition}
\newtheorem{lemma}[thm]{Lemma}

\newtheorem{lem}[thm]{Lemma}
\newtheorem{sublem}[thm]{Sublemma}
\newtheorem{cor}[thm]{Corollary}

\theoremstyle{definition}
\newtheorem{defn}[thm]{Definition}

\newtheorem{definition*}{Definition}
\newtheorem*{acknowledgements}{Acknowledgements}

\newtheorem{ex}[thm]{Example}
\theoremstyle{remark}

\newcounter{step}[section]

\newcommand{\bQ}{\mathbb{Q}}

\newcommand{\bC}{\mathbb{C}}

\newcommand{\cF}{\mathcal{F}}

\newcommand{\cH}{\mathcal{H}}

\newcommand{\cM}{\mathcal{M}}

\newcommand{\cQ}{\mathcal{Q}}

\newcommand{\fA}{\mathbf{A}}
\newcommand{\fB}{\mathbf{B}}
\newcommand{\fC}{\mathbf{C}}
\newcommand{\fD}{\mathbf{D}}

\newcommand{\rk}{\mathrm{rank}\,}

\def\GL{\mathrm{GL}}

\def\Col{\operatorname{Col}}

\renewcommand{\bar}[1]{\overline{#1}}

\usepackage{amscd,amssymb,amsmath}
\usepackage{color}

\usepackage{mleftright}

\title[Rank 2 in $\cH(6)$]%
{Rank $2$ Affine Invariant Subvarieties in $\cH(6)$}%

\author{Pramana Saldin}
\address{Department of Mathematics, University of Wisconsin--Madison, Madison, WI 53706}
\curraddr{Department of Mathematics, 970 Evans Hall, MC 3840, University of California, Berkeley, CA 94720-3840}
\email{saldin@berkeley.edu}

\author{Ruocheng Yang}
\address{Department of Mathematics, University of Wisconsin--Madison, Madison, WI 53706}
\curraddr{Department of Mathematics, University of Toronto, Toronto, ON M5S 2E4, Canada}
\email{ruocheng.yang@mail.utoronto.ca}


\begin{document}

\begin{abstract}
We classify rank $2$ rel $0$ arithmetic affine invariant subvarieties in the minimal stratum $\mathcal H(6)$. The proof follows and extends Apisa's approach in genus three, reducing the analysis to the classification of rank $1$ rel $1$ cylinder rigid affine invariant subvarieties in lower-genus boundary strata. In particular, together with the classification of algebraically primitive rank $2$ rel $0$ orbit closures, this gives a complete description of rank $2$ affine invariant subvarieties in $\mathcal H(6)$.
\end{abstract}

\maketitle

\section{Introduction}

The dynamics of the $\GL^+_2(\mathbb R)$--action on strata of Abelian differentials has become a central theme in Teichmüller dynamics. Let $\kappa=(\kappa_i)$ be a partition of $2g-2$. The \textit{stratum} $\cH(\kappa)$ parametrizes pairs $(X,\omega)$, where $X$ is a compact Riemann surface of genus $g$ and $\omega$ is a holomorphic $1$-form whose zero orders are prescribed by $\kappa$.

A breakthrough of Eskin--Mirzakhani--Mohammadi, together with the algebraicity theorem of Filip, shows that orbit closures are highly rigid: every $\GL_2^+(\mathbb R)$--orbit closure is an \emph{affine invariant subvariety}, namely a properly immersed algebraic subvariety which is locally cut out by real homogeneous linear equations in period coordinates \cite{eskin2015isolation, eskin2018invariant, filip2016semisimplicity, filip2016splitting}. Understanding and classifying these orbit closures is one of the guiding problems in the field. The smallest subfield over which the coefficients of these linear equations may be chosen is called the \emph{field of definition} \cite{wright2014field}. We call an affine invariant subvariety \emph{arithmetic} if its field of definition is $\mathbb Q$.

Compactifications of strata and of affine invariant subvarieties play an important role in this classification program. Mirzakhani--Wright introduced a partial compactification and developed a boundary theory for affine invariant subvarieties, providing a powerful framework for inductive arguments \cite{mirzakhani2017boundary}. Subsequent refinements include the WYSIWYG compactification of Chen--Wright \cite{chen2021wysiwyg} and the multi-scale compactification of Bainbridge--Chen--Gendron--Grushevsky--Möller \cite{BCGGM2018compactification,bainbridge2024modulispacemultiscaledifferentials}.

A fundamental structural result of Avila--Eskin--Möller shows that, for any affine invariant subvariety $\cM$ and any $(X,\omega)\in \cM$, the projection $p:T_{(X,\omega)}\cM \longrightarrow H^1(X;\mathbb C)$ has symplectic image \cite{AvilaEskinMoller2017}. The \textit{rank} of $\cM$ is defined to be half the complex dimension of $\operatorname{im}(p)$, and the \textit{rel} of $\cM$ is defined to be $\dim_{\mathbb C}\ker(p)$.

Recent work of Apisa--Wright resolved the high-rank regime, and more importantly, their analysis suggests that the entire classification of affine invariant subvarieties of rank $\ge 2$ reduces to the remaining case of rank $2$ rel $0$ \cite{apisa2023high}. In this paper, we study this problem in the minimal stratum $\cH(6)$ of genus $4$. Since any translation surface in $\cH(6)$ has a single zero, every affine invariant subvariety in $\cH(6)$ has rel zero. Our main result is the following.

\begin{thm}\label{thm:main-rank2-h6}
Let $\cM$ be a rank $2$ rel $0$ arithmetic affine invariant subvariety in $\cH(6)$. Then $\cM$ is a full locus of branched covers.
\end{thm}

Here a \emph{full locus of branched covers} is defined as in \cite[Section 1.2]{ApisaWright2024}.

The stratum $\cH(6)$ is the minimal stratum in genus $4$ and is a natural next case after the extensively studied classification results in genera $2$ and $3$. The proof of Theorem~\ref{thm:main-rank2-h6} follows the strategy developed by Apisa in genus $3$: we analyze carefully chosen cylinder degenerations and reduce the key boundary analysis to rank $1$ rel $1$ affine invariant subvarieties in lower-genus strata \cite{apisa2024shortproofclassificationhigher}. Related work includes Apisa--Aulicino's classification of algebraically primitive orbit closures with quadratic field of definition \cite{apisa2024algebraicallyprimitiveinvariantsubvarieties}, as well as Apisa's classification of higher-rank orbit closures in hyperelliptic components \cite{Apisa2015hyp}.

Combining Theorem~\ref{thm:main-rank2-h6} with the work of Apisa--Aulicino gives the classification of rank $2$ affine invariant subvarieties in $\cH(6)$. Indeed, by Wright's field-of-definition bound \cite[Theorem~1.5]{wright2014field}, a rank $2$ affine invariant subvariety in genus $4$ has field of definition of degree at most $2$. The degree one case is the arithmetic case treated here, while the degree two case is algebraically primitive and is covered by the classification of Apisa--Aulicino.

The low-genus classification theory provides the background for this work. In genus $2$, Calta and McMullen independently discovered infinite families of rank one affine invariant subvarieties \cite{calta2004veech,mcmullen2003billiards}. Building on these ideas, McMullen subsequently classified all orbit closures in genus $2$ \cite{mcmullen2005,McM06b,mcmullen2007dynamics}.

In genus $3$, rank $2$ affine invariant subvarieties have been classified in the works of Aulicino, Nguyen, and Wright \cite{nguyen2014non,aulicino2016classification,Aulicino2015RankTA,AulicinoNguyen2020rank2}. Ygouf \cite{Ygouf2023} classified rank one rel one affine invariant subvarieties in genus $3$ strata with at most two zeros. In parallel, Apisa \cite{Apil9} gave a classification of rank one orbit closures in the hyperelliptic components $\cH^{\mathrm{hyp}}(g-1,g-1)$. Mirzakhani--Wright \cite{mirzakhani2018full} obtained a classification of full rank affine invariant subvarieties, and Winsor \cite{10.1093/imrn/rnad144} classified full rel orbit closures.

Taken together, these results provide a substantial low-genus classification framework. Theorem~\ref{thm:main-rank2-h6} advances this program to the minimal stratum $\cH(6)$, a natural first case in genus $4$.

\begin{acknowledgements}
The authors are grateful to Paul Apisa for numerous instructive conversations, for generously sharing his insights throughout the preparation of this paper, and for carefully reading earlier drafts and providing many helpful comments. They also thank David Aulicino, Sam Freedman, Kai Fu, and Alex Wright for helpful discussions. The authors gratefully acknowledge support from the NSF through grant DMS-2304840, which supported the summer research program organized by Paul Apisa that led to this paper.
\end{acknowledgements}

\section{Rank \texorpdfstring{$1.5$}{} Loci in Genus Two}

\subsection{Preliminaries}

We begin by recalling several notions concerning cylinders on translation surfaces and cylinder rigid affine invariant subvarieties.

A \emph{cylinder} on a translation surface is an isometric embedding $\phi:\mathbb R/(c\mathbb Z)\times (0,h)\longrightarrow X$ which does not extend isometrically to a strictly larger Euclidean cylinder. The constants $c$ and $h$ are called the \emph{circumference} and \emph{height} of the cylinder, respectively. The map $\phi$ extends continuously to $\mathbb R/(c\mathbb Z)\times [0,h]$, and the images of $\mathbb R/(c\mathbb Z)\times\{0\}$ and $\mathbb R/(c\mathbb Z)\times\{h\}$ are the two \emph{boundary components} of the cylinder. Each boundary component is a union of saddle connections, and their union is the boundary of the cylinder.

Let $\cM$ be an affine invariant subvariety. Two cylinders $C_1$ and $C_2$ on a surface in $\cM$ are said to be $\cM$-\emph{equivalent} if they are parallel and remain parallel under all sufficiently small deformations in $\cM$. After rotating the surface, we will often assume that the cylinders under consideration are horizontal.

Given a cylinder equivalence class $\fC$, let $\gamma_C$ denote the core curve of a cylinder $C\in\fC$, and let $h_C$ denote its height. The \emph{standard deformation} associated to $\fC$ is
\[
\sigma_{\fC}:=\sum_{C\in\fC} h_C\gamma_C^*.
\]
By the Cylinder Deformation Theorem, the standard cylinder deformation associated to an equivalence class is tangent to $\cM$ \cite{wright2015cylinder}.

We will use the following form of cylinder rigidity. An affine invariant subvariety $\cM$ is called \emph{cylinder rigid} if each cylinder equivalence class admits a partition into \emph{subequivalence classes}. More precisely, a subequivalence class consists of exactly those cylinders whose height ratios remain locally constant under deformations in $\cM$, and whose standard deformation lies in the tangent space $T_{(X,\omega)}\cM$.

In what follows, the boundary analysis will involve rank $1$ rel $1$ cylinder rigid affine invariant subvarieties in genus at most $2$.

Let $(X,\omega)$ be a translation surface. A map
\[
f:(X,\omega)\longrightarrow (\mathbb C/\Lambda,dz),
\]
where $\Lambda\subset\mathbb C$ is a lattice, is called a \emph{torus cover} if $f$ is a branched covering of Riemann surfaces and $f^*(dz)=\omega$.

Recall that rank $1$ affine invariant subvarieties with field of definition $\mathbb{Q}$ are precisely loci of branched torus covers; see Wright's description of the field of definition \cite[Theorem~1.1]{wright2014field} and \cite[Lemma~2.11 and Remark~2.13]{lanneau2017finiteness}.

\begin{defn}
A \emph{block} in $(X,\omega)$ is a connected component of the closure of a subequivalence class of cylinders.
\end{defn}

An illustration of how preimages of cylinders and their closures assemble into blocks is given in Figure~\ref{fig:exofblocks}. For the full loci of covers considered below, marked points on the covering surface map to marked points on the base. However, not every preimage of a marked point on the base is required to be marked on the covering surface. We emphasize that the notion of a block applies to all rank $1$ rel $1$ cylinder rigid affine invariant subvarieties.

\begin{figure}[htbp]
  \centering
  \includegraphics[width=0.5\linewidth]{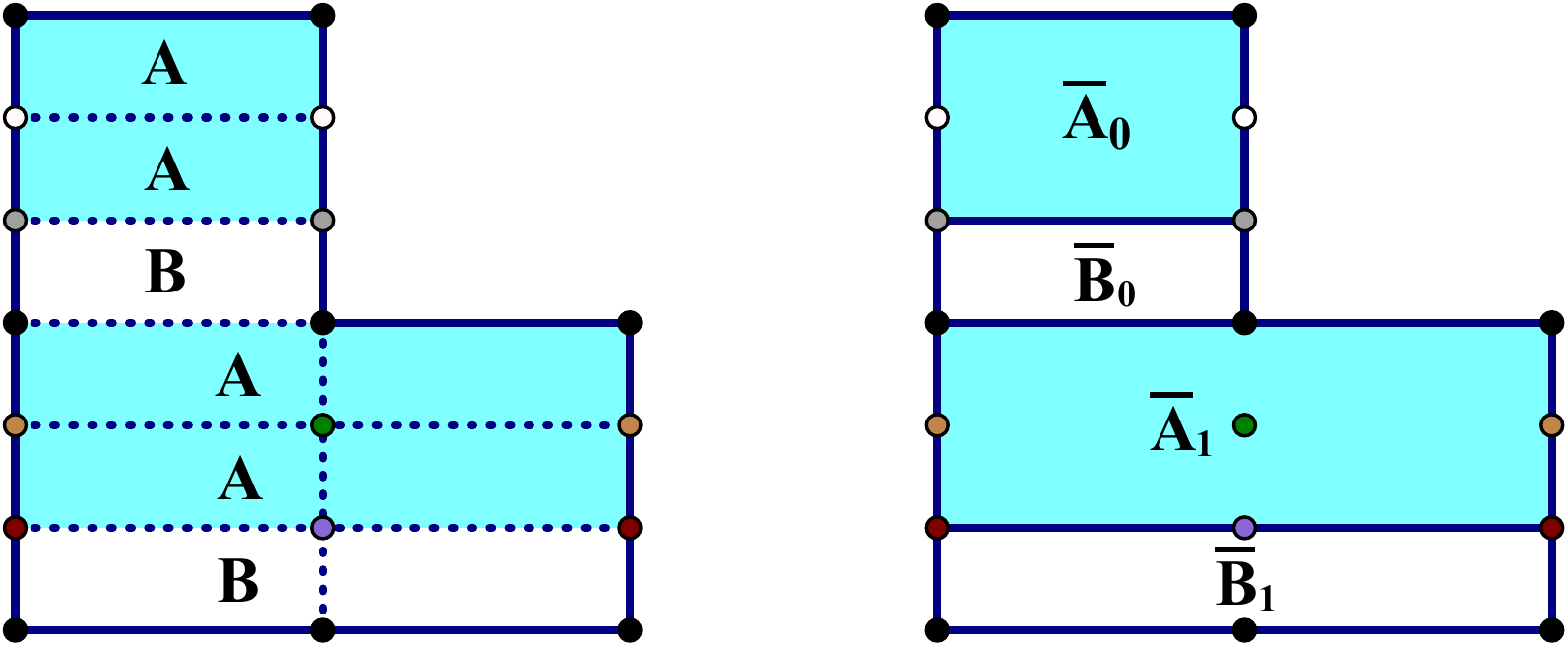}
  \caption{An example of a translation surface in $\cH(2,0^6)$. Opposite edges of the polygonal presentation are identified to obtain the surface: the black point denotes the unique zero of order~$2$, while the remaining points denote marked points. The left panel presents the surface as a $3\!:\!1$ torus cover with cylinder pattern $\mathbf{AAB}$. The right panel displays the induced subequivalence decomposition: the closures of the blue cylinders form the blocks corresponding to the class $\mathbf A$, while the closures of the white cylinders form the blocks corresponding to the class $\mathbf B$.}
  \label{fig:exofblocks}
\end{figure}

We use the terminology introduced in \cite[Section~4]{Apisa2025}. The following definition packages the properties of the boundary affine invariant subvarieties produced by the cylinder degenerations of Section~3, cf.\ \cite[Proposition~4.12]{Apisa2025}, in the form in which they will be used below.

\begin{defn}
An affine invariant subvariety $\cM$ of surfaces of genus at most $2$ is said to be \emph{rank $1$ rel $1$ cylinder rigid}, or simply \emph{rank $1.5$}, if, for every $(X,\omega)\in\cM$, the following conditions hold:
\begin{enumerate}
\item every direction on $(X,\omega)$ containing a cylinder is
completely covered by cylinders;
\item $\dim_{\mathbb C}\cM>2$;
\item in any cylinder direction, the cylinders can be partitioned
into at most two subequivalence classes, and within each
subequivalence class the ratios of cylinder heights remain constant
on all nearby surfaces in $\cM$.
\end{enumerate}
\end{defn}

Throughout this paper, a \emph{rank $1.5$ affine invariant subvariety} means a rank $1$ rel $1$ cylinder rigid affine invariant subvariety.

We make use of the notion of \emph{overcollapsing a subequivalence class}, introduced by Apisa \cite[Lemma~4.4]{apisa2021billiardsrighttrianglesorbit}. This technique plays a central role in the inductive analysis of affine invariant subvarieties. For instance, it was applied in \cite[Lemma~2]{apisa2024shortproofclassificationhigher} to classify invariant subvarieties in genus three by analyzing the rank~$1.5$ loci that arise in the genus one boundary.

Since every rank $1$ rel $1$ affine invariant subvariety with field of definition $\mathbb Q$ in $\cH(2,0^k)$ or $\cH(1,1,0^l)$ is a torus cover locus, the arithmetic rank $1.5$ loci that arise below contain dense sets of torus cover representatives. Thus, for the combinatorial arguments, it is convenient to work with rectilinear models, in which the relevant boundary saddle connections are drawn horizontally or vertically. Examples of such models are shown in Figure~\ref{F:experfectlshape}.

\begin{figure}[htbp]
\centering
\includegraphics[scale=0.3]{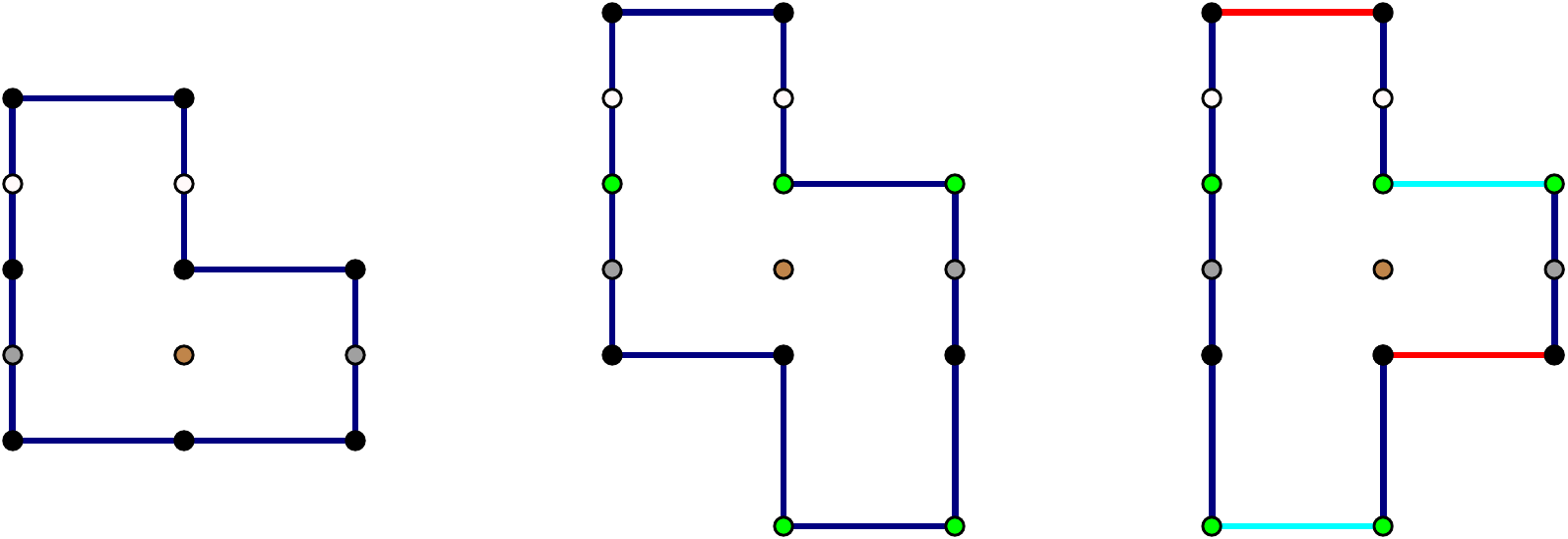}
\caption{Rectilinear representatives of the cylinder configurations used below, shown in the strata $\cH(2,0^3)$ and $\cH(1,1,0^3)$. In the last configuration, saddle connections of the same color are identified.}
\label{F:experfectlshape}
\end{figure}

The arguments below use only the induced cylinder decompositions, the subequivalence classes, and the incidence relations among the resulting blocks. Thus the rectilinear models should be viewed as convenient representatives of the relevant combinatorial configurations, rather than as additional assumptions on the orbit closure. The existence of such models in the genus two setting follows from McMullen's polygonal descriptions \cite[Section~3]{mcmullen2005}; see also \cite[Section~4.2]{Freedman2024}.

By \cite[Proposition~4.12]{Apisa2025}, the boundary components arising from the cylinder degenerations used below are rank $1.5$ affine invariant subvarieties. In this section, we analyze the low-marked cases needed in the proof of Theorem~\ref{thm:main-rank2-h6}, namely
\[
\cH(2,0),\qquad \cH(2,0^2),\qquad \cH(1,1),\qquad \cH(1,1,0).
\]
The cases $\cH(2,0^k)$ for $3\le k\le 7$ are treated in Appendix~\ref{app:rank1rel1}.

Throughout this section, we denote by $\fA$ a subequivalence class of horizontal cylinders, written as $\fA=\{A_0,\ldots,A_m\}$, ordered from top to bottom. Similarly, we denote by $\fB$ a second subequivalence class, written as $\fB=\{B_0,\ldots,B_n\}$, also ordered from top to bottom. %

\subsection{The Strata \texorpdfstring{$\cH(2,0)$}{} and
\texorpdfstring{$\cH(2,0^2)$}{}}
\label{subsec:H20-H200}
Let $(X,\omega)$ be one of the rectilinear surfaces in $\cH(2,0^k)$ considered below, and let $C_1,\ldots,C_{n_1+n_2}$ be its horizontal cylinders. In these models, the horizontal cylinders have exactly two distinct circumferences, which we denote by $c_1<c_2$. After relabeling, we assume that $C_1,\ldots,C_{n_1}$ have circumference $c_1$, and $C_{n_1+1},\ldots,C_{n_1+n_2}$ have circumference $c_2$. We refer to the former as \emph{short cylinders} and to the latter as \emph{long cylinders}.

We record a basic consequence of the existence of a rel deformation in rank $1.5$ loci in $\cH(2,0^k)$. Let $\cM\subseteq \cH(2,0^k)$ be a rank $1.5$ affine invariant subvariety containing a horizontally periodic surface, and suppose that the horizontal cylinders are partitioned into subequivalence classes $\fA$ and $\fB$. Then there exists a nontrivial rel deformation in $\cM$ whose infinitesimal effect is to rescale the heights of all cylinders in $\fA$ by a common proportional factor $d_1$, and the heights of all cylinders in $\fB$ by a possibly distinct proportional factor $d_2$.

\begin{lem}\label{lem_1or2marked}
There are no rank $1.5$ affine invariant subvarieties in the strata $\cH(2,0)$ and $\cH(2,0^2)$.
\end{lem}

\begin{proof}
Suppose first that $\cM\subset \cH(2,0)$ is rank $1.5$. Since there is only one marked point, in any horizontal cylinder decomposition one can find a single long or short cylinder whose closure contains no marked point. Its boundary consists only of saddle connections joining the unique zero to itself, so its height is unchanged under any small rel deformation moving the marked point relative to the zero. This contradicts the preceding consequence of the cylinder rigidity, namely that the nontrivial rel direction rescales the heights in each subequivalence class by a nonzero common proportional factor.

Now suppose that $\cM\subset \cH(2,0^2)$ is rank $1.5$. The two marked points cannot both lie entirely in the short cylinder region, nor can they both lie entirely in the long cylinder region. Indeed, in either case the opposite region contains a single long or short cylinder whose closure has no marked point, and the same argument gives a contradiction.

Thus, after possibly relabeling, one marked point lies in the short cylinder region and the other lies in the long cylinder region. By a small rel deformation, we may move the marked point on the long cylinder so that it lies directly below the other marked point. At the resulting surface, there is a single vertical short cylinder whose closure contains no marked point; see Figure~\ref{F:2 marked points example}. Its height is fixed under the corresponding rel deformation, again contradicting the nonzero height rescaling forced by rank $1.5$. Therefore no such rank $1.5$ affine invariant subvariety exists in $\cH(2,0^2)$.
\end{proof}

\begin{figure}[htbp]
\centering
\includegraphics[scale=0.12]{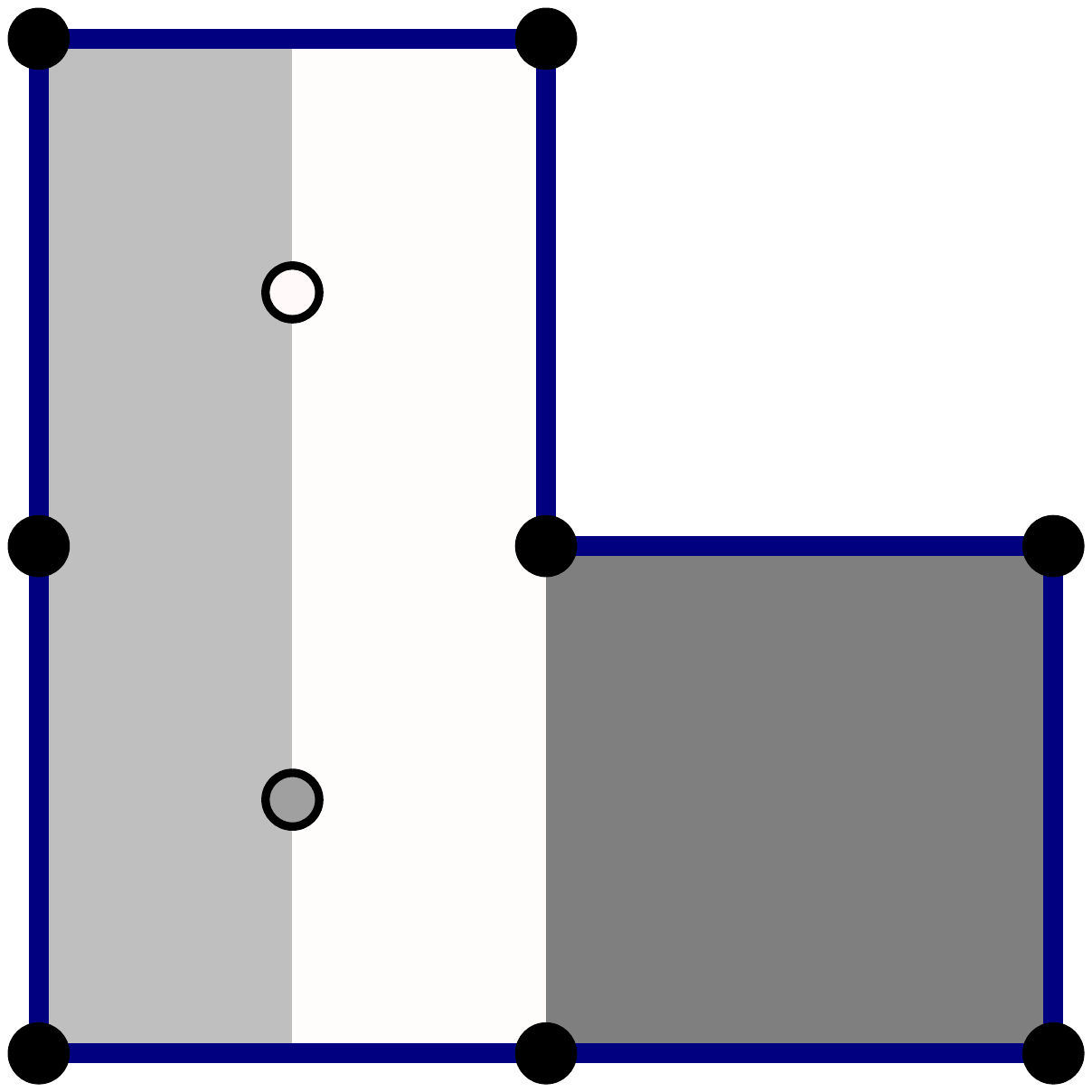}
\caption{The obstruction to rank $1.5$ loci in $\cH(2,0^2)$.}
\label{F:2 marked points example}
\end{figure}

\subsection{The Strata \texorpdfstring{$\cH(1,1)$}{} and
\texorpdfstring{$\cH(1,1,0)$}{}}
\label{subsec:H11-H110}

Recall that overcollapsing a subequivalence class $\fA$ means collapsing the cylinders in $\fA$ to height zero along the corresponding cylinder deformation direction, and then continuing slightly past the collapse in the same direction. See Apisa \cite[Lemma~4.4]{apisa2021billiardsrighttrianglesorbit} for the precise definition.

By \cite[Theorem~7.1]{mcmullen2007dynamics}, every surface in $\cH(1,1)$ admits a slit torus construction; see also \cite[Proposition~1.16]{wright2015translation}. Moreover, if $\cM\subset \cH(1,1)$ is a rank $1.5$ affine invariant subvariety, then by \cite[Example 3.5]{Aulicino2015}, $\cM$ contains a horizontally periodic surface with three horizontal cylinders. The enumeration of cylinder diagrams in $\cH(1,1)$ shows that there is only one three-cylinder diagram in this stratum; see \cite[Section~7.1]{zorich2006flatsurfaces}. Thus the last two drawings in Figure~\ref{F:experfectlshape} should be viewed as rectilinear representatives of this same cylinder diagram.

\begin{lem}\label{lem:H11-rank15-pattern}
Let $\cM\subset \cH(1,1)$ be a rank $1.5$ affine invariant subvariety. Let $(X,\omega)\in\cM$ be a horizontally periodic surface with three horizontal cylinders, labelled $C_1,C_2,C_3$ from top to bottom. Then, up to reversing the order of the cylinders and interchanging the two subequivalence classes, the subequivalence pattern is $\fA,\ \fB,\ \fA$. Moreover, the two $\fA$-cylinders have the same height.
\end{lem}

\begin{proof}
Let $C_1,C_2,C_3$ be the horizontal cylinders, ordered from top to bottom, let $\gamma_i$ denote the core curve of $C_i$, and let $h_i$ denote the height of $C_i$. In this diagram, the space of rel deformations is one-dimensional and is spanned by $\gamma_1^*-\gamma_2^*+\gamma_3^*$. By \cite[Theorem~1.5]{mirzakhani2017boundary}, the absolute projections of the standard deformations $\sigma_{\fA}$ and $\sigma_{\fB}$ are positively proportional. Hence, for some $c>0$, $\sigma_{\fA}-c\sigma_{\fB}$ is purely relative. Therefore it is a scalar multiple of $\gamma_1^*-\gamma_2^*+\gamma_3^*$.

Write
\[
\sigma_{\fA}=\sum_{C_i\in\fA} h_i\gamma_i^*,
\qquad
\sigma_{\fB}=\sum_{C_i\in\fB} h_i\gamma_i^* .
\]
The coefficients of $\sigma_{\fA}-c\sigma_{\fB}$ are positive on the cylinders in $\fA$ and negative on the cylinders in $\fB$. Since $\gamma_1^*-\gamma_2^*+\gamma_3^*$ has sign pattern $+,-,+$, the cylinders $C_1$ and $C_3$ lie in one subequivalence class and $C_2$ lies in the other. Thus, after interchanging the labels if necessary, the pattern is $\fA\fB\fA$.

With this labeling,
\[
\sigma_{\fA}-c\sigma_{\fB}
=
h_1\gamma_1^*-c h_2\gamma_2^*+h_3\gamma_3^* .
\]
Comparing the coefficients of $\gamma_1^*$ and $\gamma_3^*$ with those of a scalar multiple of $\gamma_1^*-\gamma_2^*+\gamma_3^*$ gives $h_1=h_3$. Therefore the two $\fA$-cylinders have the same height.
\end{proof}

\begin{figure}[htbp]
\centering
\includegraphics[scale=0.15]{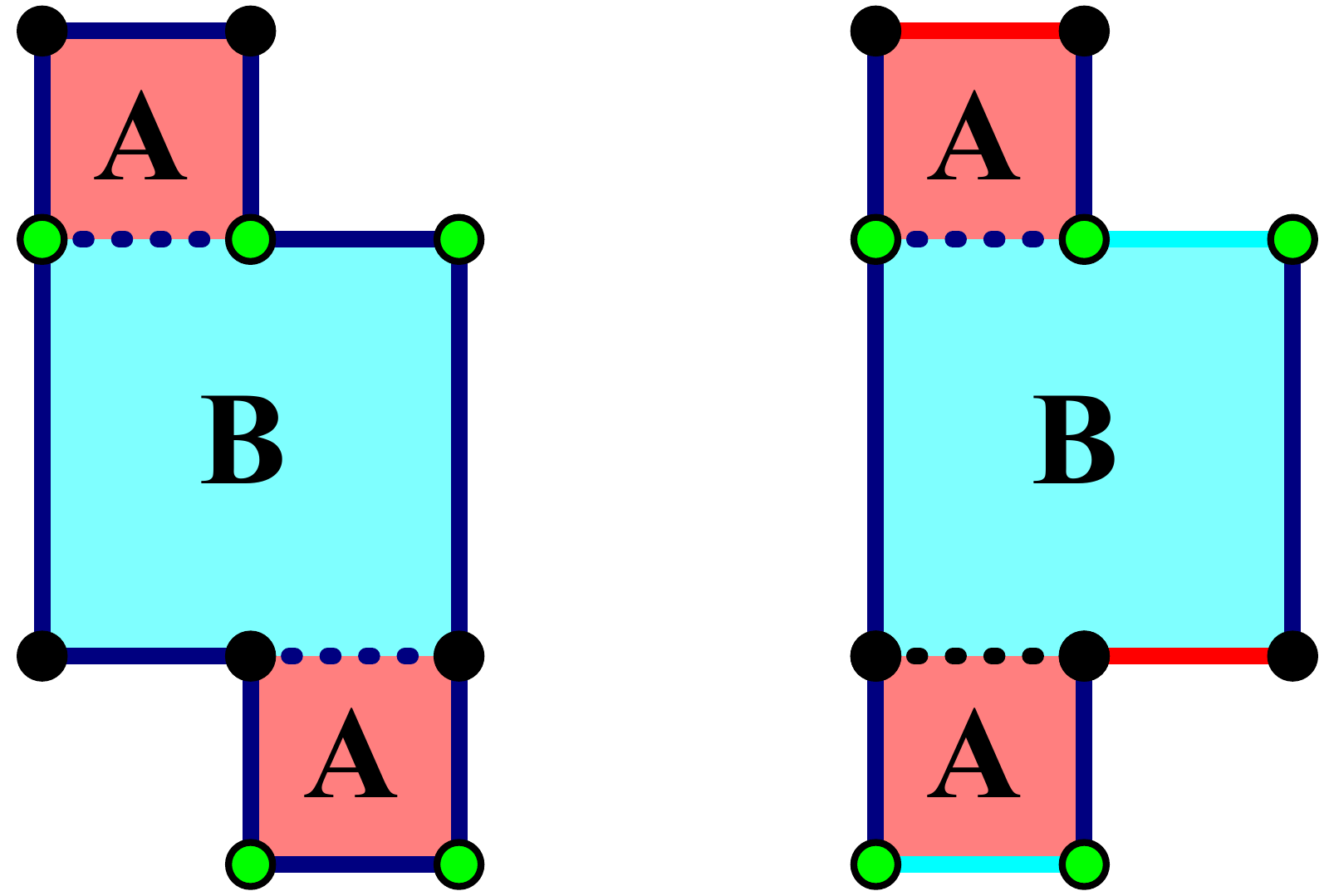}
\caption{The $\fA\fB\fA$ configuration in the standard three-cylinder diagram in $\cH(1,1)$. The two drawings differ by a standard shear of the middle cylinder and represent the same configuration for our purposes.}
\label{F:rank1.5inH(1,1)}
\end{figure}

Let $\pi:\cH(1,1,0)\to \cH(1,1)$ be the forgetful map. We first show that the rel direction of $\cM$ is not supported only at the marked point.

\begin{lem}\label{lem:marked-point-not-free}
Let $\cM\subset\cH(1,1,0)$ be a rank $1.5$ affine invariant subvariety. Then the marked point is not locally free over the underlying unmarked surface.
\end{lem}

\begin{proof}
Suppose otherwise. Then $\ker(d\pi)\cap T_{(X,\omega)}\cM$ contains a nonzero relative tangent vector supported at the marked point. Since $T_{(X,\omega)}\cM$ is complex linear, multiplying this vector by $i$ gives another tangent vector in $T_{(X,\omega)}\cM$. Thus we may move the marked point transversely to a chosen horizontal cylinder decomposition while leaving the underlying unmarked surface fixed.

By the standard classification of cylinder diagrams in $\cH(1,1)$, choose a horizontally periodic representative of the underlying surface with two horizontal cylinders. If the marked point lies on the common boundary of the two cylinders, we first move it slightly off the boundary using the transverse rel direction. Thus we may assume that the marked point lies in the interior of one horizontal cylinder. On the marked surface, this cylinder is split into two horizontal cylinders, while the other horizontal cylinder is unchanged. Moving the marked point further in the same rel direction changes the two split heights non-proportionally and leaves the remaining cylinder height fixed. This height-change pattern is incompatible with cylinder rigidity. Indeed, within a subequivalence class, height ratios must remain locally constant, so the infinitesimal height changes of cylinders in the same subequivalence class are proportional. The two split cylinders and the unchanged cylinder therefore cannot be distributed among only two subequivalence classes. This contradicts the rank $1.5$ condition.
\end{proof}

It follows that the rel direction of $\cM$ is not supported only at the marked point. Since $\cM$ has rel one, its rel tangent line has nonzero image under the derivative of the forgetful map $\pi:\cH(1,1,0)\longrightarrow\cH(1,1)$. Thus $\pi(\cM)$ has nontrivial rel. By the preceding discussion, $\pi(\cM)$ contains a horizontally periodic surface with the standard three-cylinder diagram. We may therefore work with this diagram and place the marked point in it.

\begin{lem}\label{lem:main-rank1.5}
Up to the hyperelliptic involution and interchanging the two subequivalence classes, the only arithmetic rank $1.5$ configurations in $\cH(1,1,0)$ are those shown in Figure~\ref{F:rank1rel1inH(1,1,0)}.
\end{lem}

\begin{proof}
We now consider surfaces in $\cH(1,1,0)$. By the discussion above, after forgetting the marked point we may work with the standard three-cylinder diagram in $\cH(1,1)$. Up to the hyperelliptic involution, the marked point can be placed in the three configurations shown in Figure~\ref{F:rank1rel1inH(1,1,0)blank}. We study the resulting horizontal cylinder decompositions, and we assume that the top cylinder belongs to the subequivalence class $\fA$. In each case, cylinders in the same subequivalence class are indexed from top to bottom; thus $A_0,A_1,\ldots$ denote the $\fA$-cylinders in vertical order, and similarly for $B_0,B_1,\ldots$.

\begin{figure}[htbp]
\centering
\includegraphics[scale=0.25]{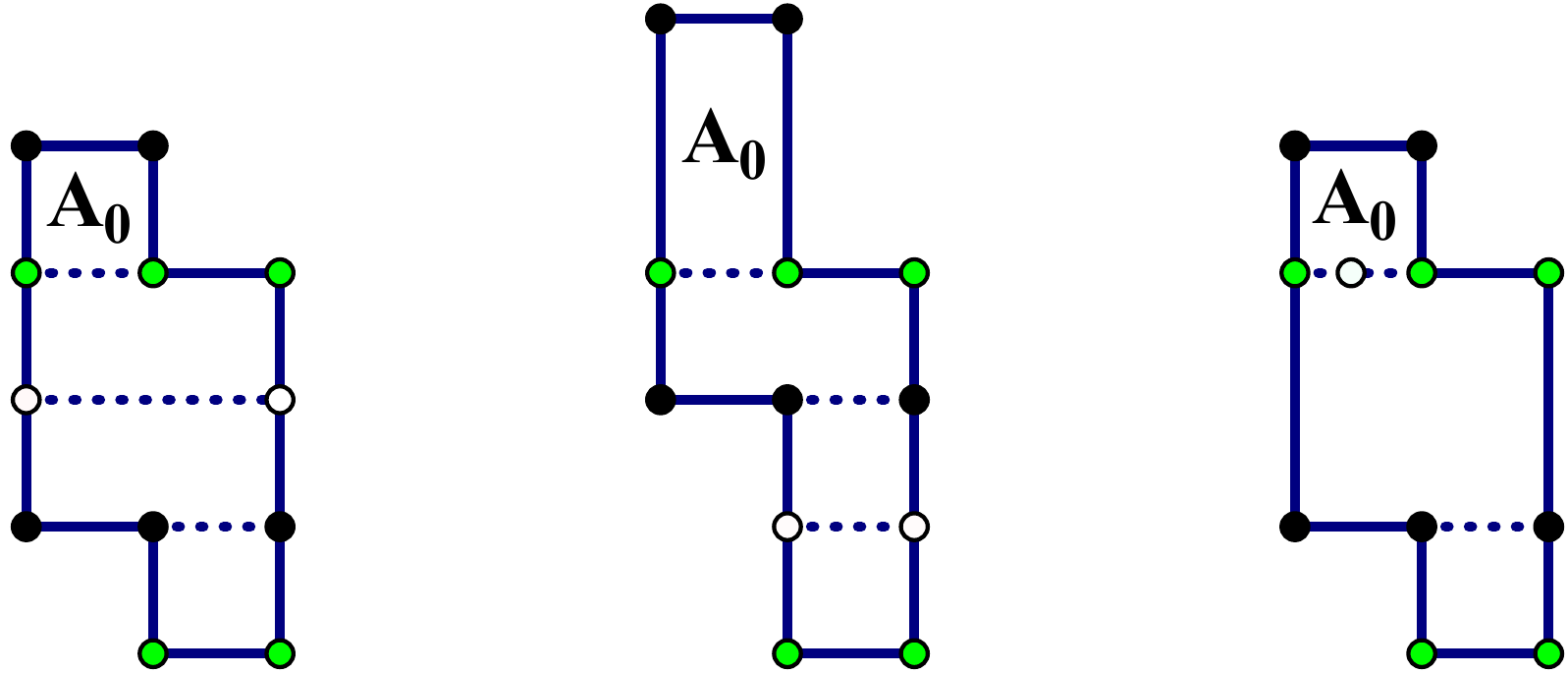}
\caption{The three configurations to consider for rank $1.5$ loci in
$\cH(1,1,0)$. The black and green points denote the two distinct simple
zeros, and the white point denotes the marked point.}
\label{F:rank1rel1inH(1,1,0)blank}
\end{figure}

Introducing the marked point divides one of the cylinders into two, so the marked surface has four horizontal cylinders. In the first configuration, the arrangements
\[
\mathbf{AAAB},\quad \mathbf{ABBB},\quad \mathbf{AABA},\quad
\mathbf{ABAA}
\]
cannot occur. Indeed, in each of these arrangements one subequivalence class consists of three cylinders and the other consists of a single cylinder. Overcollapsing the single cylinder can attack only the cylinders adjacent to it, and hence can change the heights of at most two cylinders in the three-cylinder class. Therefore at least one cylinder in the three-cylinder class has unchanged height, while another one has nonzero height change. This contradicts cylinder rigidity, since height ratios within a subequivalence class must remain locally constant. This is the same overcollapsing obstruction as in \cite[Lemma~2]{apisa2024shortproofclassificationhigher}.

The arrangements $\mathbf{AABB}$ and $\mathbf{ABAB}$ are equivalent under the hyperelliptic involution, and are also impossible. Indeed, as shown in Figure~\ref{F:rank1rel1inH(1,1,0) AABB}, after applying the relevant cylinder deformation and then overcollapsing, the cylinder $A_1$ changes height while $A_0$ does not. Hence the ratio of their heights cannot remain locally constant, contradicting cylinder rigidity.

\begin{figure}[htbp]
\centering
\includegraphics[scale=0.1]{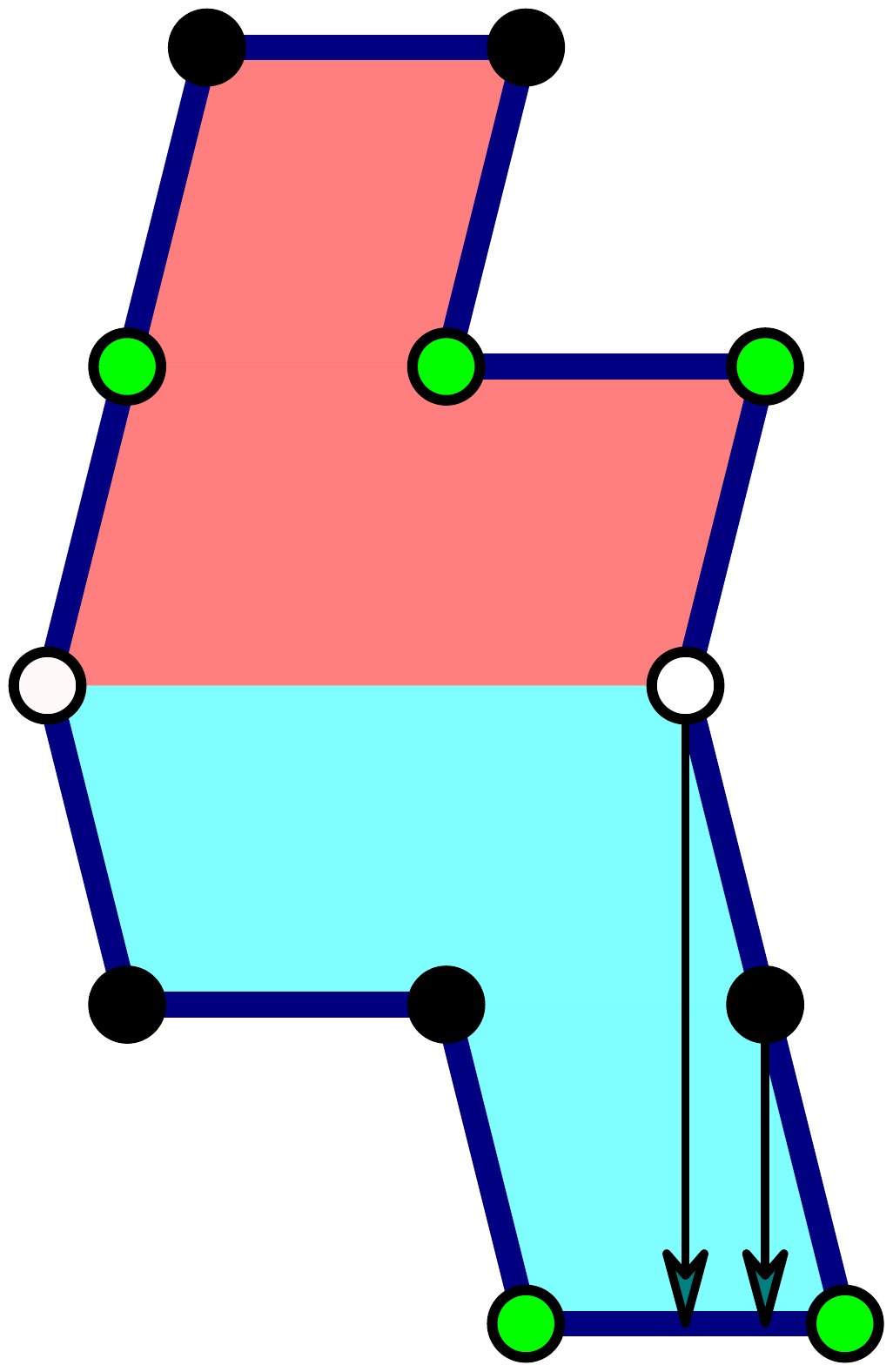}
\caption{The configuration $\mathbf{AABB}$ is impossible: after a suitable rel deformation, the height of $A_0$ remains unchanged under the subsequent small overcollapsing, while the other $\fA$-cylinder changes height.}
\label{F:rank1rel1inH(1,1,0) AABB}
\end{figure}

Therefore, in the first configuration, the only remaining possibility is $\mathbf{ABBA}$. Applying the overcollapsing algorithm in this case shows that the two $\fA$-cylinders have equal heights, and likewise the two $\fB$-cylinders have equal heights. This gives the first surviving configuration.

We next consider the second configuration in Figure~\ref{F:rank1rel1inH(1,1,0)blank}. The arrangements $\mathbf{ABBB}$ and $\mathbf{AAAB}$ cannot occur; up to the hyperelliptic involution, $\mathbf{AAAB}$ is equivalent to $\mathbf{AABA}$. In each case, the single cylinder in one subequivalence class can overcollapse into only part of the cylinders in the other subequivalence class. The arrangement $\mathbf{AABB}$ is also impossible, because one of the $\fA$-cylinders is not affected by any $\fB$-cylinder.

The arrangements $\mathbf{ABBA}$ and $\mathbf{ABAB}$ lead to irrational height solutions. We indicate the computation for $\mathbf{ABBA}$; the case $\mathbf{ABAB}$ is equivalent by the hyperelliptic involution. Let $h_{A_0}$ and $h_{A_1}$ be the heights of the two $\fA$-cylinders and normalize $h_{A_0}=1$. The overcollapsing equations reduce to
\[
\frac{1+\max\{1,h_{A_1}\}}{h_{A_1}}=\frac{1}{h_{A_1}-1}.
\]
The right-hand side is positive, so necessarily $h_{A_1}>1$. Hence $\max\{1,h_{A_1}\}=h_{A_1}$, and the equation becomes $(1+h_{A_1})/h_{A_1}=1/(h_{A_1}-1)$. Equivalently,
\[
h_{A_1}^2-h_{A_1}-1=0.
\]
Thus the only positive solution is $h_{A_1}=(1+\sqrt5)/2$. Thus \(h_{A_1}/h_{A_0}=(1+\sqrt5)/2\). This is the nonarithmetic golden-ratio configuration classified in \cite[Proposition~4.1]{apisa2024algebraicallyprimitiveinvariantsubvarieties}. Since we are classifying arithmetic rank $1.5$ loci, this case is excluded.

Hence the only remaining possibility in the second configuration is $\mathbf{ABAA}$. In this case the cylinder heights satisfy $h_{A_0}=2h_{A_1}=2h_{A_2}$. This gives the second surviving configuration.

\begin{figure}[htbp]
\centering
\includegraphics[scale=0.15]{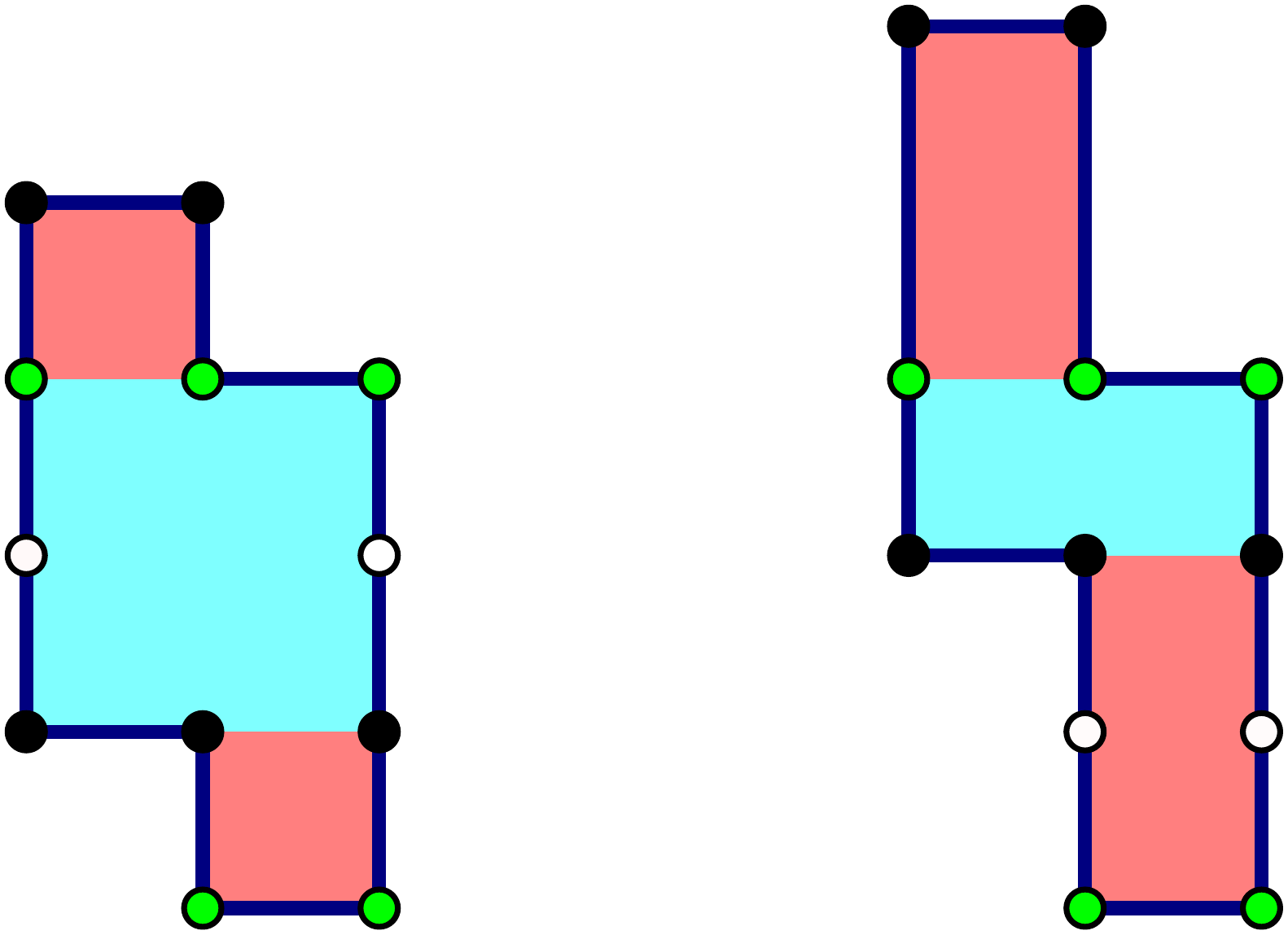}
\caption{The two arithmetic rank $1.5$ examples in $\cH(1,1,0)$, where the red cylinders lie in the subequivalence class $\fA$ and the blue cylinders lie in the subequivalence class $\fB$.}
\label{F:rank1rel1inH(1,1,0)}
\end{figure}

We now show that the marked point cannot be placed in any other position on the standard three-cylinder diagram.

\begin{sublem}\label{sublem:marked-point-position-H110}
In the first two surviving configurations, the position of the marked point is forced as in Figure~\ref{F:rank1rel1inH(1,1,0)}. In particular, the third configuration in Figure~\ref{F:rank1rel1inH(1,1,0)blank} cannot occur.
\end{sublem}

\begin{proof}
By Lemma~\ref{lem:marked-point-not-free}, the marked point is not locally free over the underlying unmarked surface. Hence any allowable rel deformation is inherited from the rank $1$ rel $1$ deformation of the underlying surface in $\cH(1,1)$.

Suppose that the marked point is placed in a position different from the two surviving configurations above. Using the subequivalence classes determined above, apply a small rel deformation as in Figure~\ref{F:marked point not move H(1,1,0)}. Under this deformation, one horizontal cylinder increases in height, two horizontal cylinders decrease in height, and the remaining horizontal cylinder has fixed height.

\begin{figure}[htbp]
\centering
\includegraphics[scale=0.15]{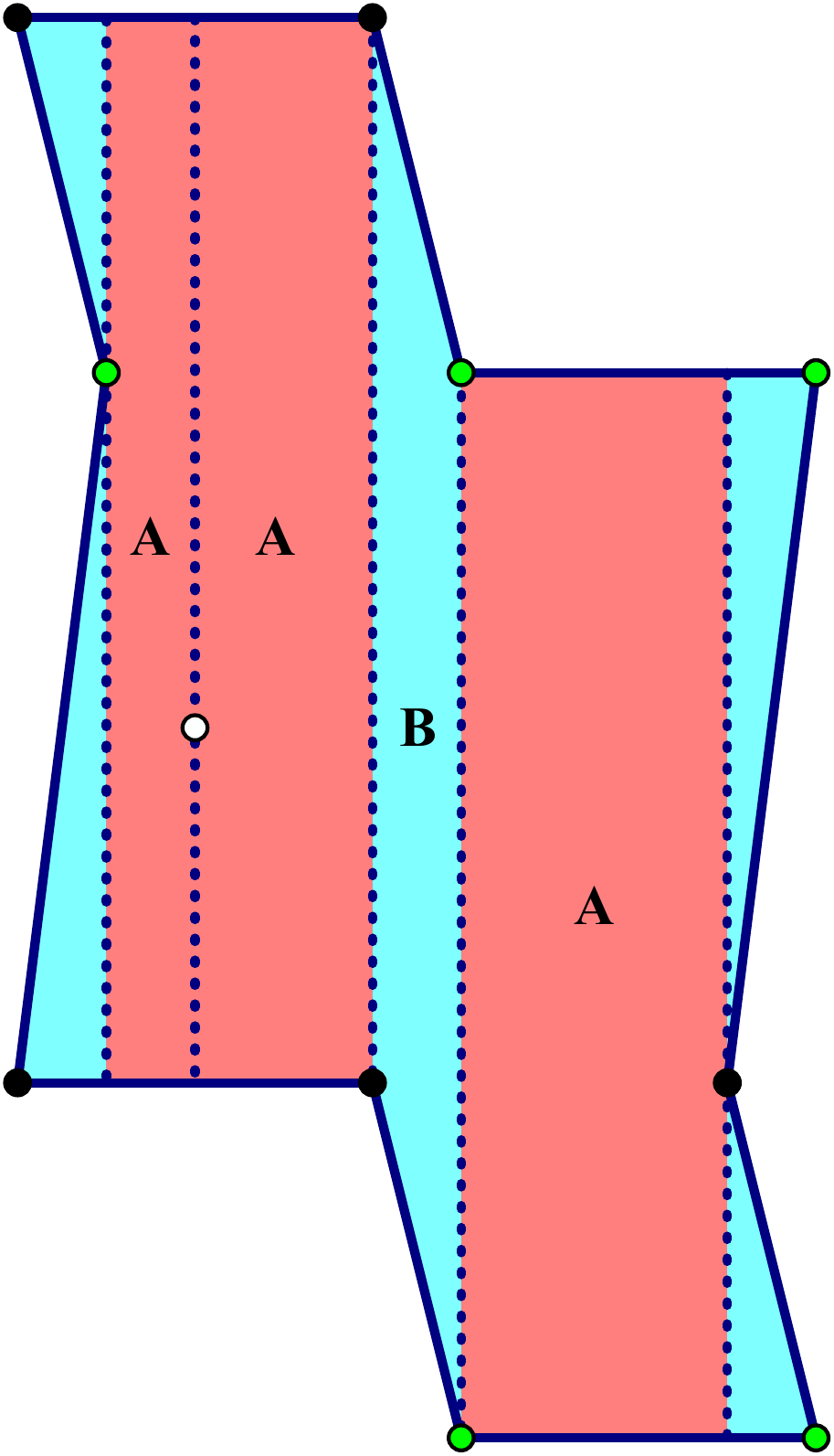}
\caption{Moving the marked point away from the positions shown in
Figure~\ref{F:rank1rel1inH(1,1,0)} produces a height-change pattern
incompatible with cylinder rigidity.}
\label{F:marked point not move H(1,1,0)}
\end{figure}

This height-change pattern is incompatible with cylinder rigidity. Indeed, within a subequivalence class, height ratios must remain locally constant, so the infinitesimal height changes of cylinders in the same subequivalence class are proportional. Thus cylinders whose heights increase, decrease, and remain fixed cannot be distributed among only two subequivalence classes. This contradicts the rank $1.5$ condition, which allows at most two subequivalence classes in any cylinder direction.

Therefore the marked point cannot be placed in any position other than the two surviving positions shown in Figure~\ref{F:rank1rel1inH(1,1,0)}. In particular, the third configuration in Figure~\ref{F:rank1rel1inH(1,1,0)blank} cannot occur.
\end{proof}

The preceding exclusions show that the only possible configurations are those shown in Figure~\ref{F:rank1rel1inH(1,1,0)}.
\end{proof}

\section{Proof of Theorem~\ref{thm:main-rank2-h6}}

\subsection{Rank 2 Invariant Subvarieties in $\cH(6)$}
By the classification of algebraically primitive affine invariant subvarieties with quadratic field of definition \cite{apisa2024algebraicallyprimitiveinvariantsubvarieties}, the only non-arithmetic rank $2$ rel $0$ affine invariant subvariety in genus $4$ is the EMMW $(1,1,1,7)$-locus in $\cH(6)$ \cite{eskin2020billiards}. Thus, in order to prove Theorem~\ref{thm:main-rank2-h6}, it remains to analyze the arithmetic case, i.e. the case where the field of definition of $\cM$ is $\mathbb Q$.

We recall the Prototype Lemma \cite[Lemma~1.9.2]{apisa2018dynamics}:

\begin{lem}\label{lem:prototype}
Let $(X,\omega)\in\cM\subset\cH(\kappa)$ where $\cM$ is a rank $2$ rel $0$ affine invariant subvariety, and let $\fC_1$ be an equivalence class of cylinders on $(X,\omega)$. Up to replacing $(X,\omega)$ by a nearby surface in $\cM$, one may assume that
\begin{enumerate}
 \item $(X,\omega)$ is horizontally periodic, with the horizontal cylinder equivalence classes denoted $\fC_1,\fC_2$;
 \item $(X,\omega)$ is vertically periodic, with the vertical cylinder equivalence classes denoted $\fD_1,\fD_2$;
 \item the inclusions $\fC_1\subseteq\fD_1$ and $\fD_2\subseteq\fC_2$ hold.
\end{enumerate}
\end{lem}

This lemma reduces the analysis to a single prototype surface in the orbit closure, which carries two distinct cylinder equivalence classes in each of the horizontal and vertical directions, related by the inclusions above.

\begin{defn}
A cylinder equivalence class $\fC$ is said to be \emph{generic} if every saddle connection on the boundary of any cylinder in $\fC$ is generically parallel to the core curves of $\fC$, and if $\fC$ attains the maximal number of cylinders locally.
\end{defn}

The prototype may be chosen so that the equivalence classes $\fC_1$ and $\fD_2$ are generic. No genericity assumption is made for the remaining equivalence classes.

Given a cylinder equivalence class $\mathbf{C}$ on a translation surface $(X,\omega)$, define the standard cylinder dilation $a_t^{\mathbf{C}}(X,\omega)$ as follows: first rotate the surface so that the cylinders in $\mathbf{C}$ are horizontal, then apply
\[
a_t=
\begin{pmatrix}
1 & 0\\
0 & e^t
\end{pmatrix}
\]
to the cylinders in $\mathbf{C}$ alone, and finally apply the inverse rotation. We then define
\[
\Col_{\mathbf{C}}(X,\omega)=\lim_{t\to -\infty} a_t^{\mathbf{C}}(X,\omega).
\]
Geometrically, this operation collapses the cylinders in $\mathbf{C}$ in the direction transverse to their core curves, and we consider the case in which this collapse produces a degenerate surface. The limit is taken in the partial compactification of \cite{mirzakhani2017boundary,chen2021wysiwyg}.

Now suppose that $(X,\omega)$ lies in an affine invariant subvariety $\mathcal M$. Then by \cite{wright2015cylinder}, $a_t^{\mathbf{C}}(X,\omega)\in \mathcal M$ for all $t\in \mathbb R$. It follows that the limit $\Col_{\mathbf{C}}(X,\omega)$ is contained in a boundary affine invariant subvariety, which we denote by $\mathcal M_{\mathbf{C}}$. For each generic equivalence class $\fC$ considered below, the rel-zero assumption and \cite[Corollary~5.2, Remark~6.6, and Lemma~6.5]{apisa2023high} imply that all cylinders in $\fC$ collapse simultaneously and that
\[
\dim_{\bC}\cM_{\fC}=\dim_{\bC}\cM-1.
\]
Moreover, \cite[Corollary~4.26]{apisa2023high} gives $\rk(\cM_{\fC})=\rk(\cM)-1$. Hence, a codimension one boundary component of a rank $2$ rel $0$ orbit closure has rank $1$ rel $1$, and is cylinder rigid by \cite[Proposition 4.12]{Apisa2025}.

\begin{defn}
A cylinder $C$ is \emph{free} if every cylinder deformation supported on $C$ remains in the ambient invariant subvariety $\mathcal{M}$.
\end{defn}

With the notation of Lemma \ref{lem:prototype}, we distinguish two cases:
\begin{itemize}
\item \textbf{Case A:} Every cylinder in the complement of $\mathbf{C}_1$ is free, and the same holds on all sufficiently nearby surfaces in $\mathcal{M}$;
\item \textbf{Case B:} $|\fC_1|=|\fD_2|=2$.
\end{itemize}

These two cases exhaust all possibilities. Indeed, suppose that Case~A does not hold. If every cylinder on every surface in $\cM$ were free, then Case~A would hold trivially; hence some surface in $\cM$ carries a non-free cylinder equivalence class, which we denote by $\fC_1$, and $|\fC_1|\ge 2$. Since Case~A fails, after replacing $(X,\omega)$ by a nearby surface in $\cM$, the complement of $\fC_1$ contains a non-free equivalence class $\fD_2$, so $|\fD_2|\ge 2$. The cylinders of $\fC_1\cup\fD_2$ are pairwise disjoint, and a surface in $\cH(6)$ admits at most four pairwise disjoint cylinders since the core curves of disjoint cylinders span an isotropic subspace of $H_1(X;\mathbb{R})$, and these core curves are linearly independent by the Gauss--Bonnet observation below. Hence $|\fC_1|+|\fD_2|\le 4$, and so $|\fC_1|=|\fD_2|=2$. Finally, after rotating so that $\fC_1$ is horizontal and $\fD_2$ is vertical, the Prototype Lemma applies to the pair $(\fC_1,\fD_2)$, as in the proof of \cite[Proposition~4]{apisa2024shortproofclassificationhigher}, and yields a prototype as in Lemma~\ref{lem:prototype} whose distinguished classes are $\fC_1$ and $\fD_2$. Therefore we are in Case~B.

These two cases will be analyzed separately below.

\subsubsection{Proving Case A}
Following \cite[Section~7]{apisa2023high}, let $C$ be a cylinder on a translation surface $(X,\omega)$ contained in the affine invariant subvariety $\mathcal{M}$. We say a cylinder $C$ is \emph{nested} in a cylinder $D$ if $C\subset \overline{D}$ and $C$ crosses $D$ exactly once. If $C$ is nested in $D$, then there is a saddle connection appearing in both the top and bottom boundaries of $D$ crossing $D$ exactly once, which serves as a \textit{cross curve} for $C$.

The following result is the Nested Free Cylinder Theorem \cite[Theorem~7.2]{apisa2023high}.

\begin{lem}\label{lem_nestedfree}
Let $\mathcal{M}$ be an affine invariant subvariety with rel $0$. If $(X,\omega)\in\mathcal{M}$ admits a nested free cylinder, then $\mathcal{M}$ is either a full stratum of Abelian differentials or a quadratic double.
\end{lem}

This lemma immediately yields the proof of Case A.

\begin{prop}\label{prop_caseA}
If $\cM$ is rank $2$ rel $0$ in $\cH(6)$ and is in Case A, then $\cM$ is either a full stratum of Abelian differentials or a quadratic double. %
\end{prop}

\begin{proof}Let $(X,\omega)$ be the prototype given by Lemma~\ref{lem:prototype},
taken sufficiently close to the original surface so that, by Case~A, every cylinder in the complement of $\fC_1$ remains free. In particular, every cylinder in $\fC_2$ is free, so each forms its own equivalence class. Since $\fC_2$ is itself an equivalence class, it follows that $\fC_2$ consists of a single cylinder.

The containment $\fD_2\subseteq\fC_2$, together with the fact that $\fD_2$ is disjoint from $\fC_1$, implies that some saddle connection appears on both the top and bottom boundaries of $C$. After shearing $C$, we may align these two occurrences. They then determine a simple cylinder $D\subseteq\overline C$ whose core curve crosses $C$ exactly once. Since $D$ lies in the complement of $\fC_1$, it is free by Case~A. Thus $D$ is a nested free cylinder, and Lemma~\ref{lem_nestedfree} gives the result.
\end{proof}

\subsubsection{Proving Case B}

Let $(X,\omega)\in \mathcal M$ be a prototype as in Lemma \ref{lem:prototype}. Recall that a cylinder is called \emph{simple} if each of its boundary components consists of a single saddle connection.%

Our first goal is to show that collapsing the equivalence class $\fC_1$ or $\fD_2$ yields a surface of genus at most two. We begin with the following auxiliary lemma.

\begin{lem}\label{lem_nohsearinc}
Suppose that $\mathcal M$ contains a translation surface $(X,\omega)$ satisfying the hypotheses of Lemma~\ref{lem:prototype}, and assume that the equivalence class $\fC_1$ consists of two cylinders $C_1$ and $C_2$. After shearing $\fC_1$, we may arrange that $C_1$ admits a vertical cross curve, and that there exists a vertical saddle connection $s\subset \bar\fC_1$ intersecting the core curve of $C_2$.
\end{lem}

\begin{proof}
Shear $\fC_1$ so that $C_1$ admits a vertical cross curve. Suppose, toward a contradiction, that no vertical saddle connection contained in $\bar\fC_1$ intersects the core curve of $C_2$. We show that this situation cannot occur.

We use $\gamma_{C_i}$ to represent the core curve of $C_i$. Choose the direction
\[
v = - i( h_{C_1}\gamma_{C_1}^* + h_{C_2}\gamma_{C_2}^*),
\]
and consider the collapse $\Col_v(\fC_1)$. By the Cylinder Degeneration Dichotomy \cite[Theorem~5.1]{apisa2023high}, since this collapse reduces the rank, $\Col_v(\fC_1)$ must be rel-scalable; in particular, its directed graph of saddle connections cannot contain any absolute cycles.

However, since collapsing $C_2$ does not lose any saddle connection, a core curve of $C_2$ does not cross a vanishing cycle, so $\Col_{\fC_1}(\gamma_{C_2})$ remains a simple closed curve, and we get an absolute cycle that appears in $\Col_v(\fC_1)$. This contradiction shows that the configuration in Figure \ref{F:cylindercannotshear} cannot arise, and hence after shearing $\fC_1$, we may arrange that $C_1$ admits a vertical cross curve, and that there exists a vertical saddle connection $s\subset\overline{\fC_1}$ intersecting the core curve of $C_2$.

\begin{figure}[htbp]
\centering
\includegraphics[scale=0.3]{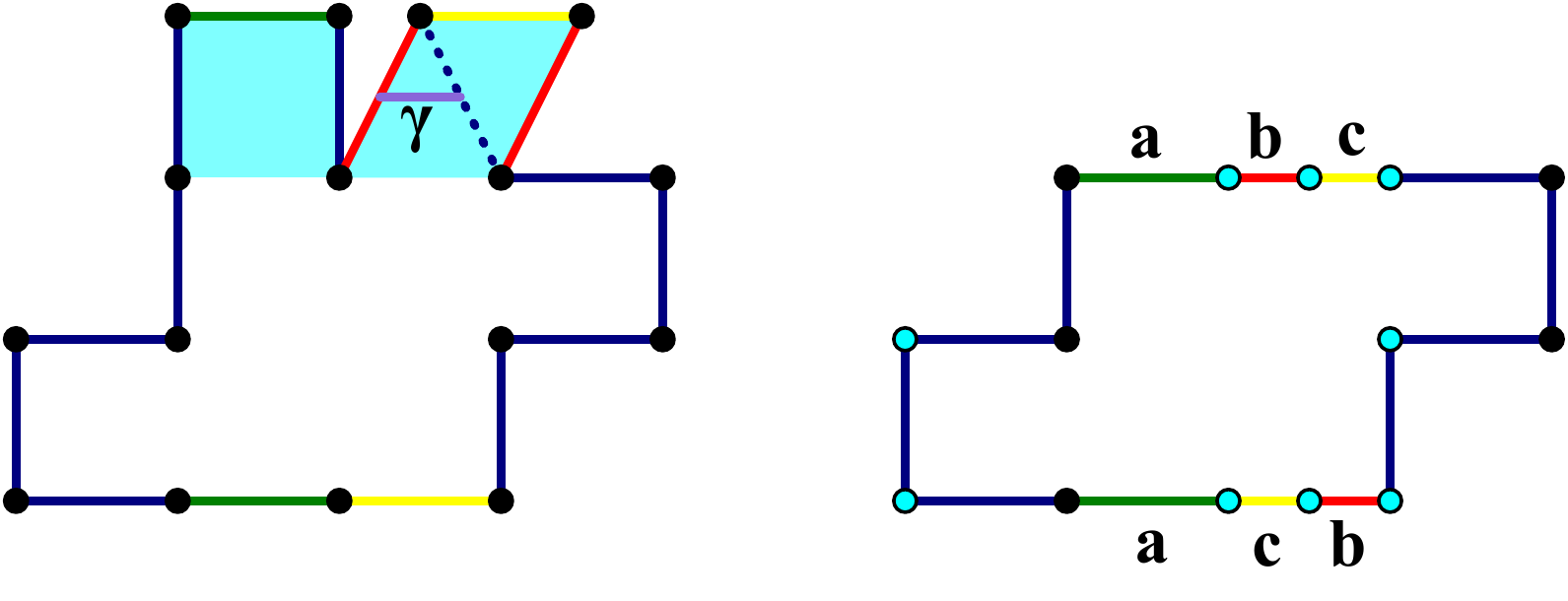}
\caption{Configuration in Lemma \ref{lem_nohsearinc}. The union of the saddle connections labelled $b$ and $c$ forms such an absolute cycle.}
\label{F:cylindercannotshear}
\end{figure}
\end{proof}

For a translation surface $(X,\omega)$, we denote by $g(X,\omega)$ its genus. If $(X,\omega)$ is disconnected, then $g(X,\omega)$ is defined to be the sum of the genera of its connected components.

We will use the following consequence of Gauss--Bonnet. On a surface in $\cH(6)$, the core curves of pairwise disjoint distinct cylinders are linearly independent in $H_1(X;\mathbb R)$. Indeed, otherwise some nonempty subcollection of these curves would bound a proper subsurface. Since $X$ has a unique zero, one component of the complement contains no zero. Its boundary consists of closed geodesics and has no corners, so Gauss--Bonnet gives $\chi=0$. As this component has nonempty boundary, it must be an annulus. This would force two of the original cylinders to lie in the same maximal cylinder, a contradiction. In particular, the core curves of two distinct cylinders cannot be homologous.

\begin{lem}\label{lem_existclassc}
Let $\cM\subset\cH(6)$ be a rank $2$ rel $0$ affine invariant subvariety in Case~B. Then there exists a prototype $(X,\omega)\in\cM$, with cylinder equivalence classes $\fC_1$ and $\fD_2$ as in Lemma~\ref{lem:prototype}, such that
\[
|\fC_1|=|\fD_2|=2,
\qquad
g\bigl(\Col_{\fC_1}(X,\omega)\bigr)\le 2,
\qquad
g\bigl(\Col_{\fD_2}(X,\omega)\bigr)\le 2.
\]
Moreover, at least one of the two collapsed surfaces
\[
\Col_{\fC_1}(X,\omega)
\quad\text{and}\quad
\Col_{\fD_2}(X,\omega)
\]
is connected.
\end{lem}

\begin{proof}
Assume that
\[
\fC_1=\{C_1,C_2\}.
\]
After shearing the cylinder equivalence class $\fC_1$, we may suppose that $C_1$ contains a vertical cross curve. By Lemma~\ref{lem_nohsearinc}, there then exists a vertical saddle connection crossing $C_2$.

By the Gauss--Bonnet observation above, the core curves of the two cylinders in $\fC_1$ are not homologous. It follows that the two vertical saddle connections above span a two-dimensional subspace of $H_1(X;\bC)$. Using the notation of \cite{mirzakhani2017boundary}, we deduce that the space of vanishing cycles $V$ has dimension at least two. Moreover, these vanishing cycles are isotropic with respect to the intersection form. Therefore, by the Mirzakhani--Wright boundary formula \cite[Theorem~2.9]{mirzakhani2017boundary},
\[
g(\Col_{\fC_1}(X,\omega))\le 4-2=2.
\]
The same argument applies to $\fD_2$, and hence $g(\Col_{\fD_2}(X,\omega))\le 2$.

Now suppose that $\Col_{\fD_2}(X,\omega)$ is disconnected. We claim that $\Col_{\fC_1}(X,\omega)$ is connected.

By \cite[Lemma~9.1]{apisa2023high}, collapsing $\fD_2$ yields a prime (see \cite[Definition 3.1]{apisa2023high}) boundary affine invariant subvariety $\cM_{\fD_2}$. Moreover, by \cite[Theorem~1.3]{chen2021wysiwyg}, any deformation of $\Col_{\fD_2}(X,\omega)$ inside $\cM_{\fD_2}$ that changes the period of an absolute cycle on one component must also change the periods of absolute cycles on every component.

Since $\fD_2$ is disjoint from $\fC_1$, the standard shear $\sigma_{\fC_1}$ determines a tangent vector in $T_{\Col_{\fD_2}(X,\omega)}\cM_{\fD_2}$. Consider the deformation $\Col_{\fD_2}(X,\omega)+t\sigma_{\fC_1}$ for sufficiently small real $t$. On any component containing a cylinder from $\fC_1$, this deformation changes the periods of absolute cycles. By the preceding paragraph, it follows that the periods of absolute cycles vary on every connected component. Hence every connected component of $\Col_{\fD_2}(X,\omega)$ contains a cylinder from $\fC_1$.

We now determine the possible boundary strata. Since $g(\Col_{\fD_2}(X,\omega))\le 2$, if $\Col_{\fD_2}(X,\omega)$ is disconnected, then each connected component has genus $1$. We claim that each connected component of the boundary surface lies in a rank $1$ rel $1$ affine invariant subvariety. Indeed, this follows from the fact that $\fC_1\subseteq \fD_1$, which implies that $\Col_{\fD_2}(\fD_1)$ is supported on each component; otherwise, the original surface $(X,\omega)$ would be disconnected. By \cite[Theorem 1.5]{mirzakhani2017boundary}, $\sigma_{\Col_{\fD_2}(\fC_2)}-\sigma_{\Col_{\fD_2}(\fC_1)}$ defines a rel deformation on the boundary. Therefore each component supports nontrivial rel, and hence lies in a rank $1$ rel $1$ affine invariant subvariety.

Now $\dim \cH(6)=8$, while $\dim(\cH(0^m)\times \cH(0^n))=2+m+n$. Since $\dim V\ge 2$, the ambient boundary stratum has complex codimension at least $2$. It follows that the only remaining possibility is $\Col_{\fD_2}(X,\omega)\in \cH(0^2)\times \cH(0^2)$.

Suppose now that $\Col_{\fD_2}(X,\omega)\in \cH(0^2)\times \cH(0^2)$. Then the cylinders in $\Col_{\fD_2}(\fC_1)$ are disjoint and simple. Since $\fC_1$ is disjoint from $\fD_2$, the cylinders in $\fC_1$ remain disjoint and simple on $(X,\omega)$. Collapsing disjoint simple cylinders does not disconnect a connected surface, because each such cylinder is attached along a single saddle connection on each boundary component. Hence $\Col_{\fC_1}(X,\omega)$ is connected.
\end{proof}

Recall that if a horizontally periodic translation surface has genus $g$ and cone set $\Sigma$ (consisting of zeros together with marked points), then the number of horizontal saddle connections is
\[
N_{g,\Sigma}=2g+|\Sigma|-2.
\]
In particular, every horizontally periodic surface in $\cH(6)$ has exactly seven horizontal saddle connections.

Let $(X,\omega)$ be a horizontally periodic surface in a rank $2$ rel $0$ affine invariant subvariety in $\cH(6)$. By Lemma~\ref{lem_nohsearinc}, after rotating and shearing if necessary, one cylinder in $\fD_2$ admits a horizontal cross curve, and there is a horizontal saddle connection crossing the other. These contain at least two distinct horizontal saddle connections, and $\Col_{\fD_2}(X,\omega)$ has at most $7-2=5$ horizontal saddle connections.

In particular, for a horizontally periodic boundary surface, the possibilities for the ambient stratum are constrained by
\[
N_{g,\Sigma}=j \ \text{on } \cH(0^j),\quad
N_{g,\Sigma}=k+3 \ \text{on } \cH(2,0^k),\quad
N_{g,\Sigma}=l+4 \ \text{on } \cH(1,1,0^l).
\]
Thus, one must have
\[
j\le 5,\qquad k\le 2,\qquad l\le 1.
\]
We will refer to these as the boundary strata relevant to the present degeneration of $\cH(6)$.

\begin{lem}\label{lem_collinstra}
At least one of the degenerations $\Col_{\fD_2}(X,\omega)$ and $\Col_{\fC_1}(X,\omega)$ gives rise to a rank $1.5$ affine invariant subvariety in either $\cH(1,1,0)$ or $\cH(1,1)$.
\end{lem}

\begin{proof}
We interpret the above degeneration using the Mirzakhani--Wright partial compactification. By Lemma~\ref{lem_existclassc}, we may assume that $\Col_{\fD_2}(X,\omega)$ is connected. Moreover, by \cite[Lemma~4.2]{Apisa2025}, the surface $\Col_{\fD_2}(X,\omega)$ is cylinder rigid. Since $(X,\omega)$ is a prototype, the surface $\Col_{\fD_2}(X,\omega)$ remains both horizontally and vertically periodic.

Recall from \cite[Lemma~2]{apisa2024shortproofclassificationhigher} that if a horizontally periodic rank $1.5$ invariant subvariety of tori has two subequivalence classes of cylinders, say $\fA$ and $\fB$, then any two $\fA$-blocks (resp. $\fB$-blocks) have the same height, and each $\fA$-block (resp. $\fB$-block) consists of either one cylinder or two cylinders of equal height. In our setting, we take $\fA=\Col_{\fD_2}(\fC_1)$ and $\fB=\Col_{\fD_2}(\fC_2)$. Thus the horizontal cylinder decomposition of $\Col_{\fD_2}(X,\omega)$ is constrained by this block structure.

We now analyze the remaining boundary configurations arising from collapsing $\fD_2$. We first show that $\Col_{\fD_2}(X,\omega)\notin \cH(0^5)$. Since $\fA$ consists of two cylinders and $\fB$ is nonempty, while a torus with five marked points has at most five horizontal cylinders, there are three, four, or five horizontal cylinders.

Suppose for contradiction that $\Col_{\fD_2}(X,\omega)\in \cH(0^5)$. Then $\Col_{\fD_2}(X,\omega)$ is a torus with five marked points. Since $\fD_2$ is vertical, collapsing $\fD_2$ produces a union of vertical saddle connections on $\Col_{\fD_2}(X,\omega)$. In particular, the only marked points that can disappear after regluing $\fD_2$ are the endpoints of the saddle connections in $\Col_{\fD_2}(\fD_2)$.

Suppose first that $\Col_{\fD_2}(X,\omega)$ has four horizontal cylinders. Then at least one of these horizontal cylinders has a boundary component containing at least two marked points.

We claim that the two cylinders in $\fA$ cannot be adjacent.

Indeed, suppose that two adjacent cylinders belong to $\fA$. Then the two $\fB$-cylinders are also adjacent. Consider the marked point on the common boundary component of the two cylinders in $\Col_{\fD_2}(\fC_1)$. Since $\Col_{\fD_2}(\fD_2)$ is a union of vertical saddle connections contained in $\Col_{\fD_2}(\fC_2)$, this marked point is not an endpoint of any saddle connection in $\Col_{\fD_2}(\fD_2)$. Hence it is unaffected by regluing $\fD_2$ and remains a marked point on the reglued surface. But regluing $\fD_2$ recovers the surface $(X,\omega)\in \cH(6)$, which has a unique singularity and no marked points. This is impossible. Therefore adjacent horizontal cylinders must belong to different subequivalence classes.

It follows that the labels alternate cyclically, so without loss of generality the pattern is $\fA\fB\fA\fB$. %

Now $\Col_{\fD_2}(\fD_2)$ is a union of vertical cross curves contained in $\overline{\Col_{\fD_2}(\fC_2)}$. Thus the regluing along $\fD_2$ affects only endpoints of vertical saddle connections contained in $\overline{\Col_{\fD_2}(\fC_2)}$. In the configuration shown in Figure~\ref{F:H(0^5)4cylinders}, however, the white marked point is not an endpoint of any vertical saddle connection contained in $\overline{\Col_{\fD_2}(\fC_2)}$. Therefore this marked point is unaffected by the regluing and remains a marked point on the resulting surface. This contradicts the fact that the reglued surface lies in $\cH(6)$. Thus the four-cylinder case is impossible.

\begin{figure}[htbp]
  \centering
  \includegraphics[width=0.3\linewidth]{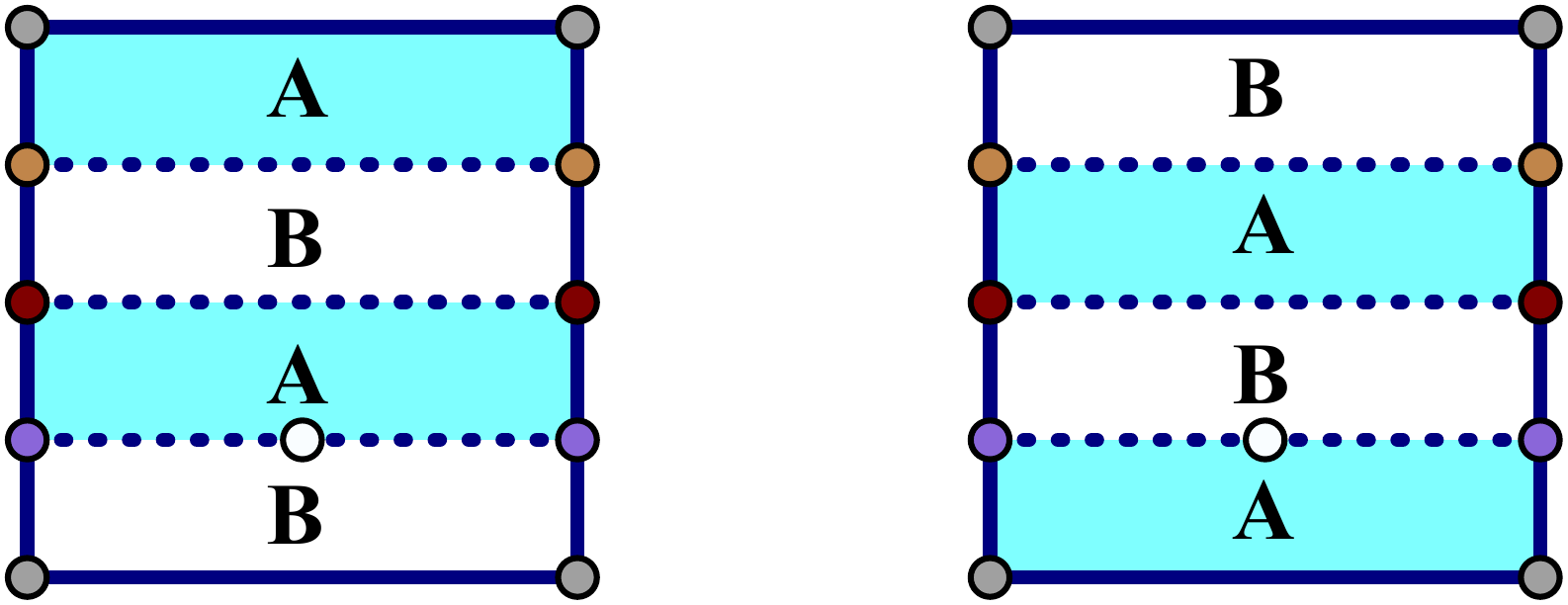}
  \caption{The four-cylinder configuration in $\cH(0^5)$.}
  \label{F:H(0^5)4cylinders}
\end{figure}

Second, suppose that $\Col_{\fD_2}(X,\omega)$ has three horizontal cylinders. Since $\fA$ consists of two cylinders, its two cylinders must be adjacent in the cyclic order. The same marked-point argument as above then gives a contradiction. Thus the three-cylinder case is impossible.

It remains to consider the case in which $\Col_{\fD_2}(X,\omega)$ has five horizontal cylinders. On a torus, these cylinders are cyclically separated by five horizontal boundary levels, each of which contains at least one marked point. Since $\Col_{\fD_2}(X,\omega)$ has exactly five marked points, each boundary level contains exactly one marked point. Consequently, every horizontal cylinder on $\Col_{\fD_2}(X,\omega)$ is simple. In particular, the two cylinders in $\fA=\Col_{\fD_2}(\fC_1)$ are simple.

Since $\fC_1$ is disjoint from $\fD_2$, every boundary saddle connection of a cylinder in $\fC_1$ persists under the collapse of $\fD_2$. Therefore the two cylinders in $\fC_1$ were already simple on $(X,\omega)$. Write $\fC_1=\{C_1,C_2\}$ and let $\gamma_i$ be the core curve of $C_i$. After applying the standard shear supplied by Lemma~\ref{lem_nohsearinc}, choose transverse cross curves $\delta_i\subseteq\overline{C_i}$. Since $(X,\omega)\in\cH(6)$ has a unique zero, each $\delta_i$ is an absolute cycle. After choosing orientations, $\langle\delta_i,\gamma_j\rangle=\delta_{ij}$. These intersection relations show that both the cross curves $\delta_1,\delta_2$ and the core curves $\gamma_1,\gamma_2$ are linearly independent. In particular, the union of the two core curves is nonseparating. Hence $\Col_{\fC_1}(X,\omega)$ is connected and the genus drops by two. Moreover, since $C_1$ and $C_2$ are simple, the vanishing-cycle space is spanned by $\delta_1$ and $\delta_2$, so the ambient stratum loses exactly two complex dimensions.

It follows that $\Col_{\fC_1}(X,\omega)$ is a connected genus two surface lying in an ambient stratum of complex dimension $8-2=6$. Therefore $\Col_{\fC_1}(X,\omega) \in\cH(2,0^2)\cup\cH(1,1,0)$. The boundary affine invariant subvariety is rank $1.5$, so Lemma~\ref{lem_1or2marked} excludes the first possibility. Hence
\[
\Col_{\fC_1}(X,\omega)\in\cH(1,1,0).
\]

The class $\Col_{\fC_1}(\fD_2)$ still consists of two cylinders. After rotating by $\pi/2$, we apply the classification of rank $1.5$ configurations in $\cH(1,1,0)$. In the second configuration in Figure~\ref{F:rank1rel1inH(1,1,0)}, the corresponding subequivalence class contains either one or three cylinders. Thus $\Col_{\fC_1}(X,\omega)$ must be the first configuration, in which the two cylinders in $\Col_{\fC_1}(\fD_2)$ are simple. Since $\fC_1$ is disjoint from $\fD_2$, the two cylinders in $\fD_2$ were already simple on $(X,\omega)$.

Applying the same cross-curve argument to the two cylinders in $\fD_2$ shows that $\Col_{\fD_2}(X,\omega)$ is connected of genus two. This contradicts the assumption that $\Col_{\fD_2}(X,\omega)$ is in $\cH(0^5)$, which has genus one. Therefore the five-cylinder case is impossible, and hence $\Col_{\fD_2}(X,\omega)\notin\cH(0^5)$.

Next suppose that $\Col_{\fD_2}(X,\omega)\in\cH(0^4)$. The class $\fA:=\Col_{\fD_2}(\fC_1)$ contains two cylinders. As above, these two cylinders cannot be adjacent. Since the other horizontal subequivalence class is nonempty, $\Col_{\fD_2}(X,\omega)$ has at least three horizontal cylinders. The three-cylinder case would force the two $\fA$-cylinders to be adjacent, while a torus with four marked points has at most four horizontal cylinders. Thus $\Col_{\fD_2}(X,\omega)$ has exactly four horizontal cylinders, with cyclic pattern $\fA\fB\fA\fB$.

The four cylinders are cyclically separated by four horizontal boundary levels, each containing at least one marked point. Since $\Col_{\fD_2}(X,\omega)$ has exactly four marked points, each boundary level contains exactly one marked point. Hence all four horizontal cylinders are simple, and in particular the two cylinders in $\fA$ are simple.

As in the five-cylinder case, the prototype containments imply that the two cylinders in $\fC_1$ were already simple on $(X,\omega)$. The preceding cross-curve argument then shows that $\Col_{\fC_1}(X,\omega) \in\cH(2,0^2)\cup\cH(1,1,0)$. Lemma~\ref{lem_1or2marked} excludes the first possibility, and hence $\Col_{\fC_1}(X,\omega)\in\cH(1,1,0)$.

The class $\Col_{\fC_1}(\fD_2)$ still consists of two cylinders. By the same classification argument as above, these cylinders are simple. The prototype containments therefore imply that the two cylinders in $\fD_2$ were already simple on $(X,\omega)$. Applying the cross-curve argument to $\fD_2$ shows that $\Col_{\fD_2}(X,\omega)$ is connected of genus two and lies in an ambient stratum of complex dimension six. This contradicts $\Col_{\fD_2}(X,\omega)\in\cH(0^4)$, since $\cH(0^4)$ has genus one and complex dimension five. Therefore $\Col_{\fD_2}(X,\omega)\notin\cH(0^4)$.

If $\Col_{\fD_2}(X,\omega)\in\cH(0^3)$, then its horizontal decomposition contains exactly three cylinders: two in $\Col_{\fD_2}(\fC_1)$ and one in $\Col_{\fD_2}(\fC_2)$. The two cylinders in $\Col_{\fD_2}(\fC_1)$ must therefore be adjacent, and the same marked-point argument as above gives a contradiction. The strata $\cH(0^j)$ with $j\le2$ are impossible because the horizontal decomposition contains at least three cylinders. Hence
\[
\Col_{\fD_2}(X,\omega)\notin\cH(0^j)
\qquad\text{for all }j\le3.
\]

We have therefore excluded all genus one boundary strata. By the list of relevant boundary strata above, the remaining possibilities are
\[
\cH(2),\quad \cH(2,0),\quad \cH(2,0^2),\quad
\cH(1,1),\quad\text{and}\quad \cH(1,1,0).
\]
The stratum $\cH(2)$ has no relative tangent directions and therefore cannot contain a rank $1.5$ affine invariant subvariety. Moreover, Lemma~\ref{lem_1or2marked} excludes $\cH(2,0)$ and $\cH(2,0^2)$. Consequently,
\[
\Col_{\fD_2}(X,\omega)
\in\cH(1,1)\cup\cH(1,1,0),
\]
which proves the lemma.
\end{proof}

Given a stratum $\mathcal{Q}(\kappa)$ of quadratic differentials, we denote by $\widetilde{\mathcal{Q}}(\kappa)$ the quadratic double consisting of holonomy double covers of surfaces in $\mathcal{Q}(\kappa)$.

\begin{lem}\label{lem:H110-collapse-branched-cover}
Let $\cM\subset\cH(6)$ be an arithmetic rank $2$ rel $0$ affine invariant subvariety, and let $(X,\omega)\in\cM$ be a prototype as in Lemma~\ref{lem:prototype}. Suppose that $\Col_{\fD_2}(X,\omega)\in \cH(1,1,0)$. Then $\cM$ is the quadratic double cover locus $\widetilde{\cQ}(5,-1)$.
\end{lem}

\begin{proof}
By Lemma~\ref{lem:main-rank1.5}, the horizontal cylinder configuration of $\Col_{\fD_2}(X,\omega)$ is one of the two configurations shown in Figure~\ref{F:rank1rel1inH(1,1,0)}. Since $\fD_2$ is disjoint from $\fC_1$, the two cylinders in $\fC_1$ survive as two cylinders in $\Col_{\fD_2}(\fC_1)$. The second configuration cannot occur, since its two subequivalence classes contain one and three cylinders, respectively. Hence $\Col_{\fD_2}(X,\omega)$ has the first configuration in Figure~\ref{F:rank1rel1inH(1,1,0)}.

In particular, the two cylinders in $\fC_1$ have the same height. Let their common height be $h$, and let their circumferences be $c_1$ and $c_2$. By Lemma~\ref{lem_nohsearinc}, after applying a standard shear to $\fC_1$, we may assume that both cylinders admit vertical cross curves.

Continue the standard shear by the amount corresponding to one full Dehn twist in the first cylinder. Since both cylinders have height $h$, this shear produces horizontal displacement $c_1$ in each cylinder. The first cylinder again admits a vertical cross curve. Applying the argument of Lemma~\ref{lem_nohsearinc} at the resulting surface, the second cylinder must also admit a vertical cross curve. Hence $c_1$ is an integer multiple of $c_2$. Interchanging the two cylinders gives that $c_2$ is an integer multiple of $c_1$. Therefore $c_1=c_2$.

The first configuration in Figure~\ref{F:rank1rel1inH(1,1,0)} also implies that the two cylinders in $\fC_2$ have equal circumference on the collapsed surface. Since their circumference ratio is locally constant on $\cM$ and converges to $1$ under the collapse of $\fD_2$, the two cylinders in $\fC_2$ already have equal circumference on $(X,\omega)$.

We next collapse $\fC_1$. The two cylinders in $\fC_1$ are disjoint and simple, and their collapse does not disconnect the surface. Their cross curves determine two independent vanishing cycles: independence follows from the fact that their core curves are not homologous, by the Gauss--Bonnet observation above. Thus exactly two complex dimensions are lost. By the preceding classification of the possible genus-two boundary strata,
\[
\Col_{\fC_1}(X,\omega)\in\cH(1,1,0).
\]

After rotating by $\pi/2$, Lemma~\ref{lem:main-rank1.5} applies again. Since $\Col_{\fC_1}(\fD_2)$ consists of two cylinders, the second configuration is again excluded. Therefore $\Col_{\fC_1}(X,\omega)$ has the first configuration in Figure~\ref{F:rank1rel1inH(1,1,0)}. The two cylinders in $\fD_2$ are consequently disjoint and simple, and the equal-height relation in the rotated picture says that they have equal circumference on $(X,\omega)$.

The configuration for $\Col_{\fD_2}(X,\omega)$ specifies how the two $\fD_2$-cylinders are inserted into $\fC_2$, while the configuration for $\Col_{\fC_1}(X,\omega)$ specifies how the two $\fC_1$-cylinders are inserted into $\fD_1$. Since these cylinders are simple and occur in equal-circumference pairs, their boundary saddle connections must be identified as shown in Figure~\ref{F:col(H(6)) all cases}. The only remaining freedom is the common twist in each pair, which can be normalized by the standard shears of $\fC_1$ and $\fD_2$. Hence, up to these allowed deformations, the resulting surface is precisely the one shown in Figure~\ref{F:col(H(6)) all cases}.

\begin{figure}[htbp]
\centering
\includegraphics[scale=0.25]{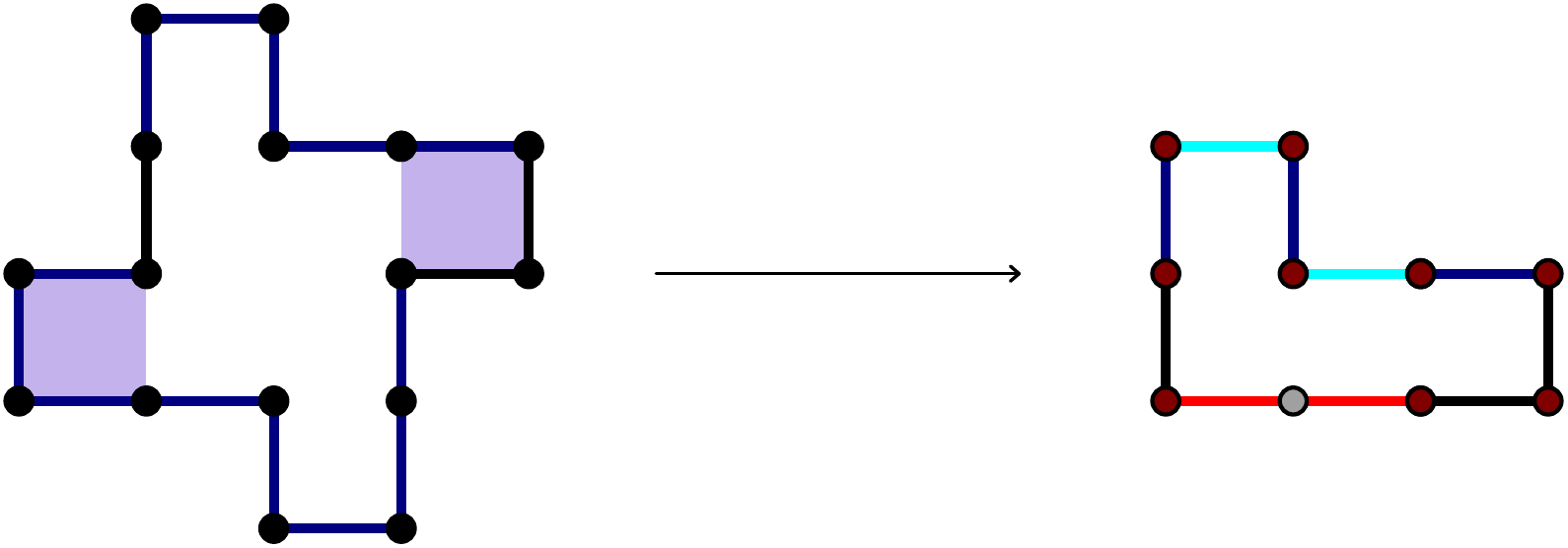}
\caption{The unique remaining configuration in the case
$\Col_{\fD_2}(X,\omega)\in \cH(1,1,0)$. The left panel shows the corresponding translation surface in $\cH(6)$, and the right panel shows the quotient half-translation surface in $\cQ(5,-1)$. Black edges are
identified with the opposite black edges by translation, while blue edges
are identified with blue edges and red edges with red edges by rotation by
$\pi$.}
\label{F:col(H(6)) all cases}
\end{figure}

This configuration admits an involution exchanging the cylinders in each equal pair, and its quotient is the half-translation surface in $\cQ(5,-1)$ shown on the right side of the figure. The equal-length and gluing relations persist locally on $\cM$, so $\cM\subseteq\widetilde{\cQ}(5,-1)$. Since both loci have complex dimension $4$, they agree on the relevant component. Therefore
\[
\cM=\widetilde{\cQ}(5,-1).
\]
\end{proof}

\begin{prop}\label{prop_caseB}
Suppose that the arithmetic affine invariant subvariety $\cM$ has rank $2$ rel $0$ in $\cH(6)$. If there exists a prototype surface $(X,\omega)\in \cM$ such that one equivalence class $\fC_1$ consists of two horizontal cylinders and the other equivalence class $\fD_2$ consists of two vertical cylinders, then $\mathcal M$ is a full locus of branched covers.
\end{prop}

\begin{proof}
By Lemma~\ref{lem_existclassc}, after rotating the surface by $\pi/2$ and relabelling the equivalence classes if necessary, we may assume that $\Col_{\fD_2}(X,\omega)$ is connected. Moreover, $g\bigl(\Col_{\fD_2}(X,\omega)\bigr)\le 2$. The argument in the proof of Lemma~\ref{lem_collinstra} then gives $\Col_{\fD_2}(X,\omega)\in \cH(1,1,0)\cup\cH(1,1)$. If $\Col_{\fD_2}(X,\omega)\in\cH(1,1,0)$, then Lemma~\ref{lem:H110-collapse-branched-cover} applies, and we are done.

Suppose instead that
\[
\Col_{\fD_2}(X,\omega)\in\cH(1,1).
\]
By Lemma~\ref{lem:H11-rank15-pattern}, the two cylinders in $\Col_{\fD_2}(\fC_1)$ are the two outer cylinders in the $\mathbf{ABA}$ configuration. In particular, they are disjoint and simple. Since $\fD_2$ is disjoint from $\fC_1$, regluing $\fD_2$ does not alter these cylinders. Hence the two cylinders in $\fC_1$ are also disjoint and simple on $(X,\omega)$, and collapsing them does not disconnect the surface.

By Lemma~\ref{lem_nohsearinc}, after applying a standard shear to $\fC_1$, we may assume that both cylinders admit vertical cross curves. These cross curves generate the vanishing cycles for the collapse of $\fC_1$. By the proof of Lemma~\ref{lem_existclassc}, they are independent. Since the two cylinders are simple, these are exactly the vanishing cycles of the collapse. Thus collapsing $\fC_1$ loses exactly two complex dimensions.

Lemma~\ref{lem_existclassc} gives $g\bigl(\Col_{\fC_1}(X,\omega)\bigr)\le 2$. Applying the genus-one and $\cH(2,0^k)$ exclusions from the proof of Lemma~\ref{lem_collinstra}, with the horizontal and vertical directions interchanged, gives $\Col_{\fC_1}(X,\omega)\in \cH(1,1,0)\cup\cH(1,1)$. Since
\[
\dim_{\bC}\cH(6)=8,\qquad
\dim_{\bC}\cH(1,1,0)=6,\qquad
\dim_{\bC}\cH(1,1)=5,
\]
a collapse losing exactly two complex dimensions cannot land in $\cH(1,1)$. Therefore $\Col_{\fC_1}(X,\omega)\in\cH(1,1,0)$.

After rotating by $\pi/2$ and interchanging the roles of $\fC_1$ and $\fD_2$, Lemma~\ref{lem:H110-collapse-branched-cover} implies that $\cM=\widetilde{\cQ}(5,-1)$. In particular, $\cM$ is a full locus of branched covers.
\end{proof}

\begin{proof}[Proof of Theorem~\ref{thm:main-rank2-h6}]
Let $\cM$ be a rank $2$ rel $0$ arithmetic affine invariant subvariety in $\cH(6)$. By Lemma~\ref{lem:prototype} we divide into two cases.

In Case~A, Proposition~\ref{prop_caseA} implies that $\cM$ is either a full stratum of Abelian differentials or a quadratic double cover locus. The full-stratum alternative is impossible, since the full stratum $\cH(6)$ has rank $4$, whereas $\cM$ has rank $2$. Hence $\cM$ is a quadratic double cover locus, and in particular is a full locus of branched covers.

In Case~B, the hypotheses of Proposition~\ref{prop_caseB} are satisfied, so $\cM$ is again a full locus of branched covers. Since the two cases exhaust all possibilities, the theorem follows.
\end{proof}

\appendix

\section{Rank \texorpdfstring{$1$}{} Rel \texorpdfstring{$1$}{}
Cylinder Rigid Loci in \texorpdfstring{$\cH(2,0^k)$}{} for \texorpdfstring{$k\le 7$}{}}
\label{app:rank1rel1}

The proof of the classification of rank two affine invariant subvarieties in $\cH(6)$ only requires the nonexistence of rank $1.5$ loci in $\cH(2,0)$ and $\cH(2,0^2)$. Here we prove the stronger, computer-assisted result that every rank $1.5$ locus in $\cH(2,0^k)$ with $k\le 7$ is a trivial torus cover; see Proposition~\ref{prop:bigh2rank15}. This range includes all cases needed for the genus four applications considered here, with $k=7$ being the extremal case arising from the principal stratum.

The computation proceeds by enumerating the possible marked-point configurations and applying an overcollapsing algorithm to each of them. The resulting attacking-speed relations involve maxima of cylinder heights. After fixing the values attained by these maxima, the relations reduce to finitely many polynomial systems, which we solve exactly using Gr\"obner bases. The remaining candidates are then eliminated by the geometric criteria described below. Every surviving configuration is a trivial torus cover.

We begin by specifying which torus covers will be regarded as trivial.

\begin{defn}
A rank $1.5$ torus cover is said to be \emph{trivial} if it is a cover of $\cH(0,0;0_{\mathrm{wp}}^{w})$, where $0_{\mathrm{wp}}^{w}$ denotes $w$ marked Weierstrass points with $0 \le w \le 4$, and if, in addition, for the covering map to the marked torus, the preimage of each block on the base is a union of blocks on the cover.
\end{defn}

We now describe the combinatorial and geometric tools used in the computation. Two cylinders are said to have the same \emph{label} if they lie in the same subequivalence class, denoted $\fA$ or $\fB$.

As before, overcollapsing a subequivalence class $\fA$ means applying a standard deformation to $\fA$ until its height becomes zero, and then continuing slightly past this collapse along the same deformation direction. The local effect on the other subequivalence class is illustrated in Figure~\ref{fig:attackfigure}. In such a configuration, overcollapsing $\fA$ may shorten saddle connections entering $\fB$ from above or from below. We define the \emph{attacking speed from above}, respectively \emph{from below}, to be the maximal rate of decrease of a saddle connection that begins at the top, respectively bottom, of $\fB$ and passes only through $\fA$.

\begin{figure}[htbp]
  \centering
  \includegraphics[width=\linewidth]{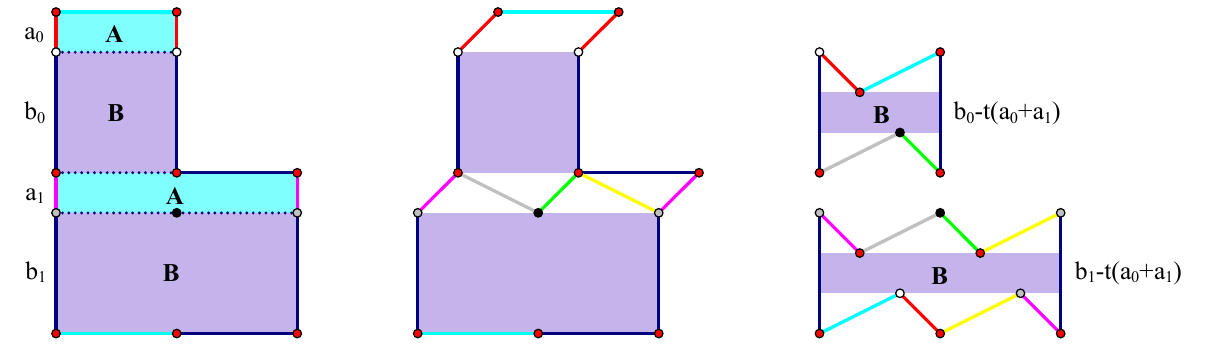}
  \caption{Overcollapsing the subequivalence class $\mathbf A$. Sides with the same color are identified, and $t$ denotes any sufficiently small positive real parameter. The first panel shows the original surface $(X,\omega)$; the middle panel shows the result of shearing only the cylinders in $\mathbf A$; the rightmost panel depicts the surface obtained after overcollapsing $\mathbf A$.}
  \label{fig:attackfigure}
\end{figure}

We use two additional mechanisms to exclude candidate configurations.

First, rel deformations may rule out configurations that are defined over $\bQ$ but are not trivial torus covers. See, for example, Figure~\ref{fig:moving-marked-points-example}. We refer to such configurations as \emph{other solutions} in the data file.

\begin{figure}[htbp]
\centering
\includegraphics[scale=0.3]{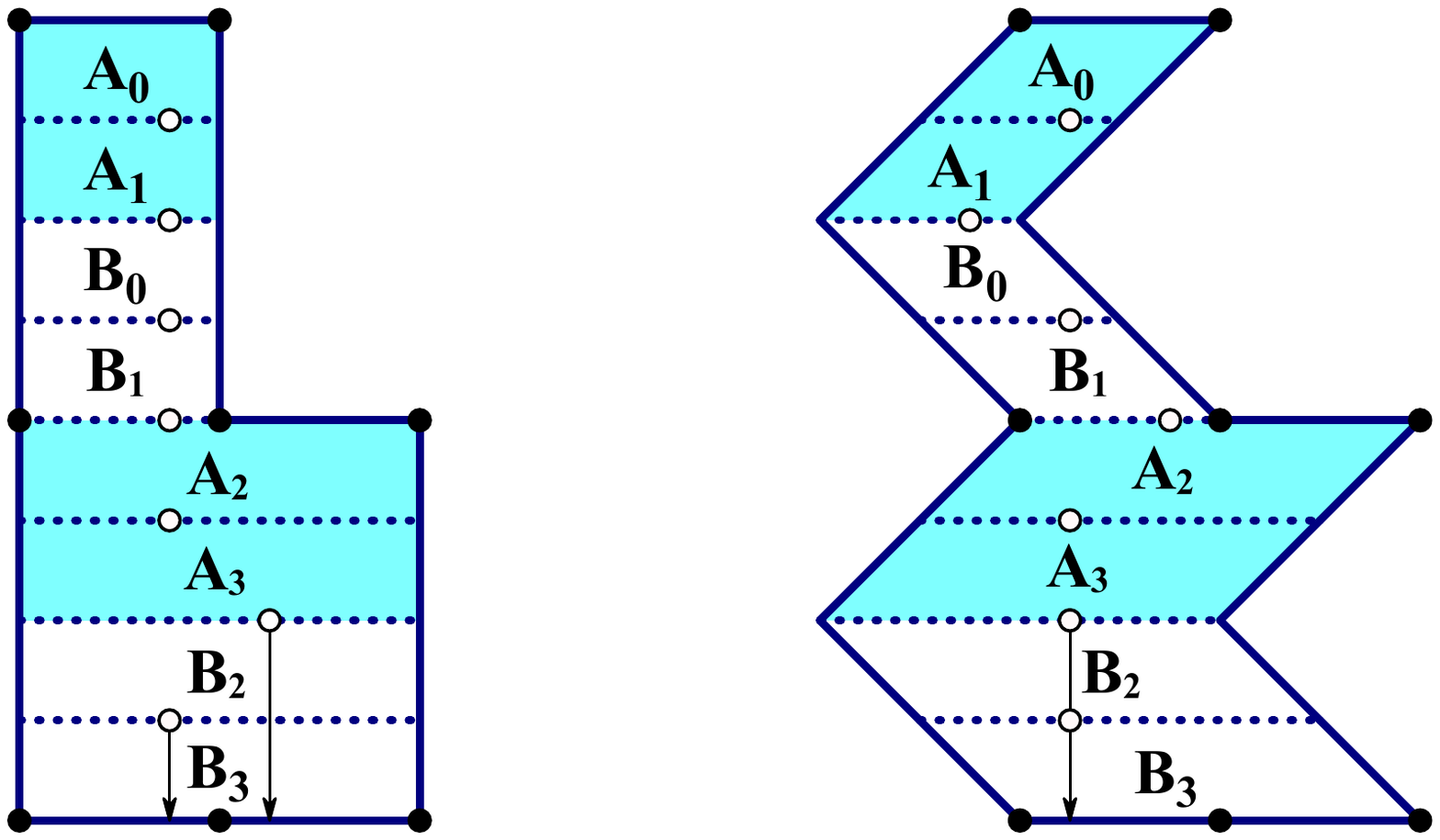}
\caption{Suppose the blue cylinders $(A_0,\dots,A_3)$ form the subequivalence class $\fA$ and the white cylinders $(B_0,\dots,B_3)$ form the subequivalence class $\fB$, with all heights equal to $1$. This configuration does not define a rank $1.5$ example, since a rel deformation allows one to move the marked point on $B_2$ so that it may collide either with $A_0$ or simultaneously with $A_0$ and $A_2$, producing distinct attacking speeds.}
\label{fig:moving-marked-points-example}
\end{figure}

Second, irrational ratios of cylinder heights are incompatible with a rank $1.5$ orbit closure in the present setting.

\begin{cor}\label{cor:no-nonarithmetic-H2}
There are no nonarithmetic rank $1.5$ affine invariant subvarieties in $\cH(2,0^k)$ for any $k\geq 1$.
\end{cor}

\begin{proof}
By \cite[Proposition~4.1]{apisa2024algebraicallyprimitiveinvariantsubvarieties}, every nonarithmetic rank $1$ rel $1$ cylinder rigid affine invariant subvariety in genus two is the golden eigenform locus with exactly one golden point marked. After forgetting the marked point, this locus is contained in $\cH(1,1)$. On the other hand, forgetting the marked points of a locus in $\cH(2,0^k)$ gives a locus in $\cH(2)$. Hence no such nonarithmetic locus can occur in $\cH(2,0^k)$.
\end{proof}

For completeness, we record the following geometric restriction on irrational height configurations. This observation is not needed for the classification below. Write $\cF(X,\omega)$ for the surface obtained from $(X,\omega)$ by forgetting all marked points.

\begin{lem}\label{lem_noirrational}
Suppose $(X,\omega)$ lies in a rank $1.5$ affine invariant subvariety and that $\cF(X,\omega)\in\cH(2)$ or $\cH(1,1)$. Assume there exist two $\fB$-cylinders on $(X,\omega)$ whose heights have irrational ratio, and that $\cF(\Col_{\fA}(X,\omega))\in\cH(2)$ or $\cH(1,1)$. Then any two $\fB$-cylinders contained in the same cylinder of $\cF(X,\omega)$ have equal height, and at most two such $\fB$-cylinders may occur in a single cylinder.
\end{lem}

\begin{proof}
By Mirzakhani--Wright \cite[Corollary~2.14]{mirzakhani2017boundary}, the collapsed surface $\Col_{\fA}(X,\omega)$ lies on a Teichm\"uller curve in the corresponding genus two stratum. The irrationality of the height ratio of the two $\fB$-cylinders rules out the torus cover case for this Teichm\"uller curve. Hence, by M\"oller's results on periodic points on Veech surfaces \cite[Theorems~5.1--5.2]{moller2006periodic}, any marked points on $\Col_{\fA}(X,\omega)$ are fixed by the hyperelliptic involution. On a genus two surface, the hyperelliptic fixed points lie on the central symmetry axis of each horizontal cylinder and therefore cut any horizontal cylinder containing them into two subcylinders of equal height. Since collapsing the $\fA$-cylinders does not alter the heights of the $\fB$-cylinders, it follows that any two $\fB$-cylinders inside the same cylinder of $\cF(X,\omega)$ must be the two halves produced by such a fixed point. In particular, their heights coincide and there are at most two of them.
\end{proof}

For the classification in this appendix, irrational candidates will instead be excluded using \cite[Proposition~4.1]{apisa2024algebraicallyprimitiveinvariantsubvarieties}, which classifies all nonarithmetic rank $1.5$ loci in genus two.

We also record two illustrative examples.

\begin{lem}
The example in Figure~\ref{fig:exofblocks} illustrates a trivial torus cover, while the example in Figure~\ref{fig:exnontrivial} gives a nontrivial torus cover.
\end{lem}

\begin{proof}
The trivial example is immediate from the definition. The second example is rank $1.5$ because it lies in the boundary of the gothic locus \cite{mcmullen2017cubic, eskin2020billiards}. It is nontrivial, since there is no torus model for which the preimage of the subequivalence class $\fA$ is again a union of subequivalence classes of the same type.
\end{proof}

\begin{figure}[htbp]
  \centering
  \includegraphics[width=0.5\linewidth]{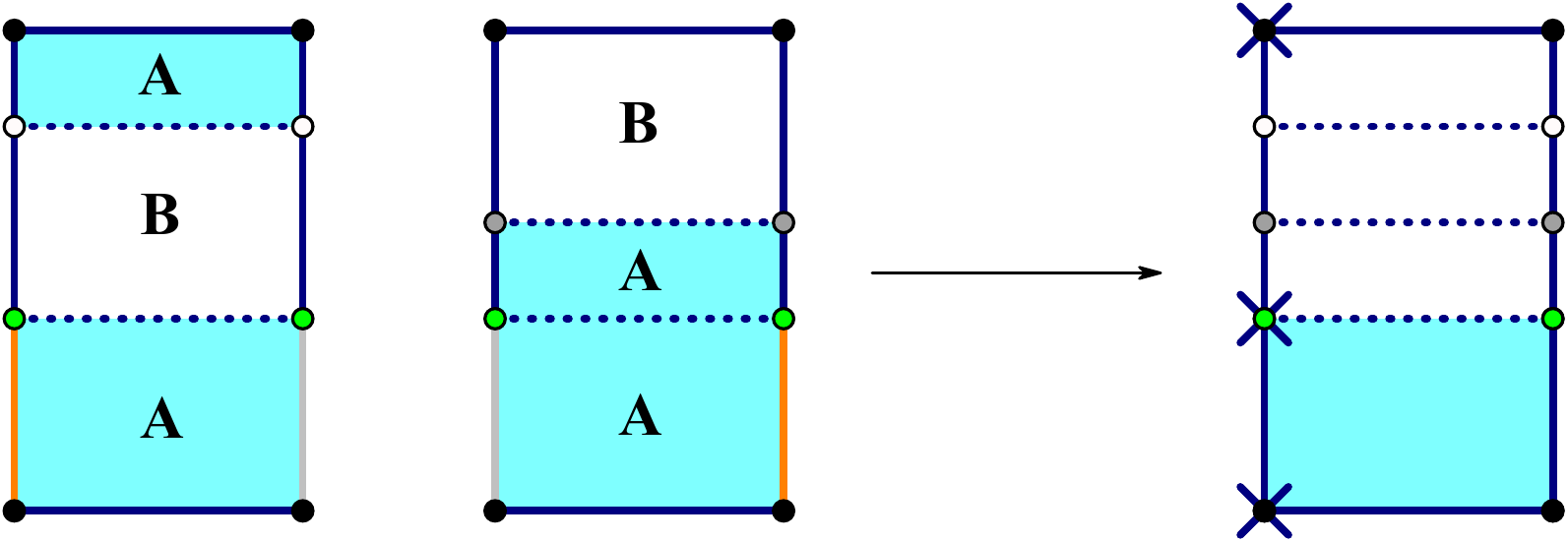}
  \caption{A translation surface in $\cH(1,1,0^2)$ which covers a surface in $\cH(0,0;0_{\mathrm{wp}}^{2})$. Opposite edges of the polygonal presentation are identified, and the grey, respectively orange, saddle connection is glued to the other grey, respectively orange, segment. The black and green points denote distinct cone points of order~$1$, while the white and grey points represent marked points.}
  \label{fig:exnontrivial}
\end{figure}

The next lemma gives a linear relation among cylinder heights that is used in the computation.

\begin{lemma}\label{lem:H2-sum-heights-equations}
Let $(X,\omega)\in\cH(2,0^k)$ be horizontally periodic and let the horizontal cylinders be partitioned into two subequivalence classes $\fA$ and $\fB$. Denote by $\fA_{\mathrm{short}}$ and $\fB_{\mathrm{short}}$ the collections of short cylinders in $\fA$ and $\fB$, and by $\fA_{\mathrm{long}}$ and $\fB_{\mathrm{long}}$ the corresponding collections of long cylinders. Then
\[
\left(\sum_{A\in\fA_{\mathrm{short}}}h(A)\right)
\left(\sum_{B\in\fB_{\mathrm{long}}}h(B)\right)
=
\left(\sum_{A\in\fA_{\mathrm{long}}}h(A)\right)
\left(\sum_{B\in\fB_{\mathrm{short}}}h(B)\right),
\]
where $h(\cdot)$ denotes the height of a cylinder.
\end{lemma}

\begin{proof}
The short and long blocks satisfy the same linear relation coming from the rel deformation. Since the rel deformation preserves absolute periods, its effect on the short and long regions is governed by the same coefficients. More precisely, if $a_1$ and $a_2$ are these coefficients, then
\[
a_1\sum_{A\in\fA_{\mathrm{short}}} h(A)+a_2\sum_{B\in\fB_{\mathrm{short}}} h(B)
=
\sum_{A\in\fA_{\mathrm{short}}} h(A)+\sum_{B\in\fB_{\mathrm{short}}} h(B),
\]
and the same relation holds for the long cylinders. Comparing the two equations yields the stated identity.
\end{proof}

\begin{cor}\label{cor:normalization-H2}
In the setting above, both subequivalence classes occur in the long and short regions. Hence we may choose reference cylinders $A_0\in\fA$ and $B_0\in\fB$. By independently rescaling the two subequivalence classes using their standard cylinder deformations, we normalize $h(A_0)=h(B_0)=1$.
\end{cor}

\begin{proof}
If one of the two regions were entirely made of cylinders from a single subequivalence class, then the identity in Lemma~\ref{lem:H2-sum-heights-equations} would force the other region to have the same property. This would leave only one subequivalence class, contrary to the setup. Thus both classes occur in the relevant regions.
\end{proof}

We now turn to the main result of this appendix.

\begin{prop}\label{prop:bigh2rank15}
Every rank $1.5$ affine invariant subvariety in $\cH(2,0^k)$ with $k \le 7$ is a trivial torus cover.
\end{prop}

\begin{proof}
For each $1\le k\le 7$, we enumerate all possible marked point configurations in $\cH(2,0^k)$ compatible with the block decomposition above. For each candidate, the overcollapsing procedure gives a finite collection of polynomial systems in the cylinder heights, after normalizing $h(A_0)=h(B_0)=1$. The maxima appearing in the attacking speeds are handled by splitting into finitely many cases, so each case is an exact polynomial computation.

The resulting solutions fall into three types. Some systems have no positive solutions. Some systems have solutions with irrational height ratios, and these are excluded by Corollary~\ref{cor:no-nonarithmetic-H2}. The remaining rational solutions are checked against the rel deformation test described above: if moving a marked point produces two different attacking speeds, then the candidate cannot define a rank $1.5$ orbit closure. After these exclusions, every surviving configuration is a trivial torus cover in the sense of the definition.

The enumeration and the exact polynomial checks are available at
\[
\url{https://github.com/pr4-kp/rank-1-5-solver}.
\]
The organized output is contained in the file \texttt{h2/H2\_Solutions\_web\_new.xlsx}.
\end{proof}

\begin{ex}[The computation in $\cH(2,0^3)$]\label{ex:rank15-H2000}
We illustrate the computation in the stratum $\cH(2,0^3)$. An exhaustive enumeration subject to Corollary \ref{cor:normalization-H2} gives the configurations shown in Figure~\ref{fig:proof-H2000}, labeled $J_1,\ldots,J_8$ and $K_1,\ldots,K_6$. Here and below, when no confusion is possible, we write $A_i$ and $B_i$ for the heights of the corresponding cylinders.

\begin{figure}[htbp]
\centering
\includegraphics[scale=0.5]{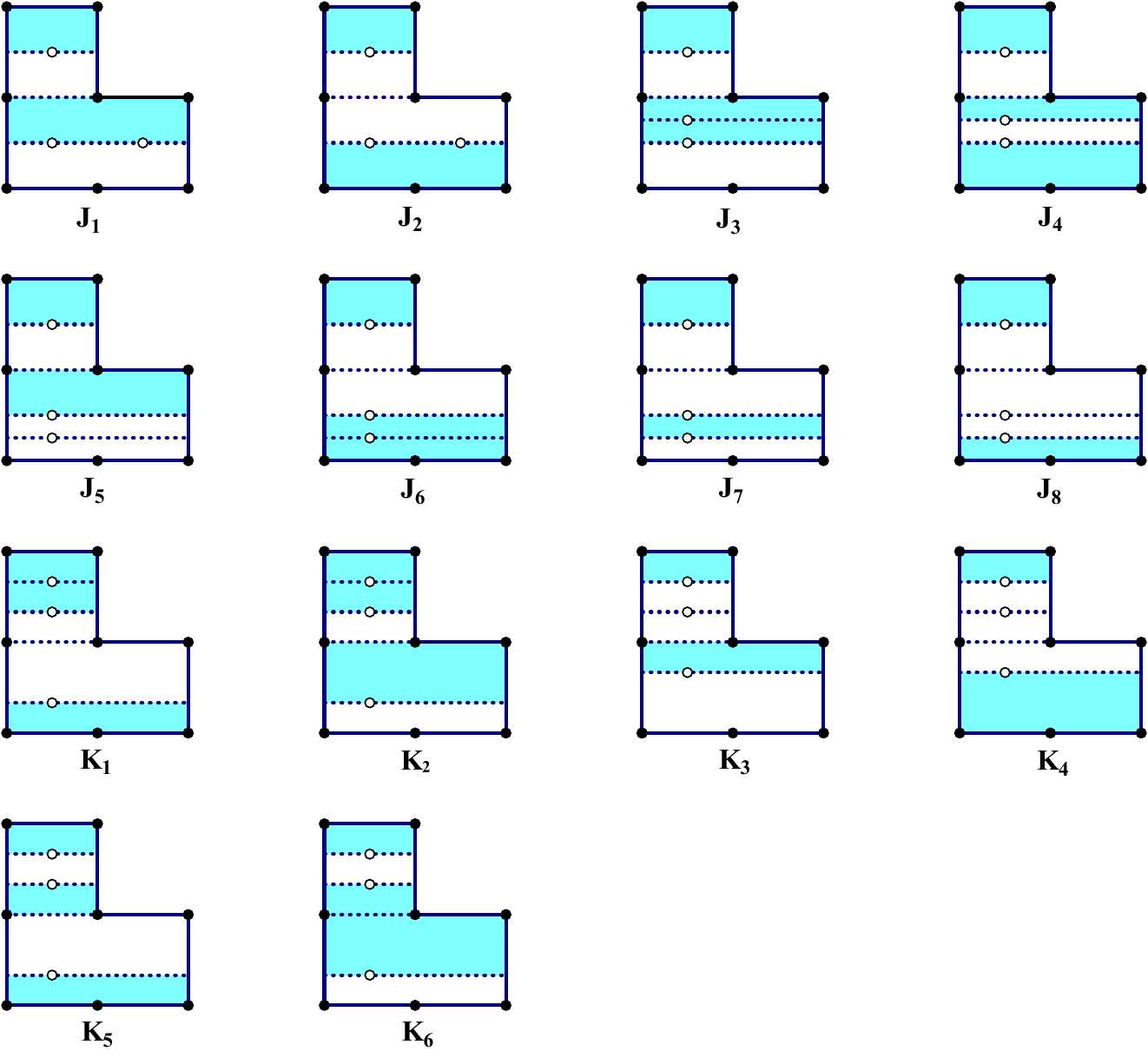}
\caption{Suppose the surface has three marked points in a rank $1.5$ orbit closure. The configurations shown above exhaust all possibilities. Blue cylinders correspond to the subequivalence class $\fA$, while white cylinders correspond to the subequivalence class $\fB$.}
\label{fig:proof-H2000}
\end{figure}

We emphasize that for a long cylinder, the marked point may lie on either the left or the right side. This may affect whether it can attack the topmost long cylinder. Although this possibility is not illustrated in the figure, it is included in the case analysis.

We apply the overcollapsing algorithm to the configurations in Figure~\ref{fig:proof-H2000}. The cases $J_6$ and $J_8$ are excluded by the rel deformation test described above: moving the relevant marked point produces two different possible attacking speeds, so the configuration cannot define a rank $1.5$ orbit closure. The case $J_2$ gives a positive solution with an irrational ratio of cylinder heights, and is therefore excluded by Corollary~\ref{cor:no-nonarithmetic-H2}. The cases
\[
J_1,\quad J_3,\quad J_5,\quad K_2,\quad K_3
\]
give trivial torus covers. The remaining cases
\[
J_4,\quad J_7,\quad K_1,\quad K_4,\quad K_5,\quad K_6
\]
have no positive solution to the corresponding overcollapsing equations.

We give the calculation for the case $K_2$; the other surviving cases are similar. In this case the cylinders occur in the order $A_0, A_1, B_0, A_2, B_1$. After normalizing $A_0=B_0=1$, the overcollapsing equations are
\[
\begin{aligned}
A_0B_0-A_1B_1&=0,\\
A_0(B_0+B_1)-A_2B_1&=0,\\
B_0(A_0+A_1+A_2)-B_1(A_0+A_1+A_2)&=0,\\
-A_2B_0+B_1(A_0+A_1)&=0.
\end{aligned}
\]
Since all heights are positive, the third equation gives $B_1=B_0=1$. The first equation then gives $A_1=1$, and the second equation gives $A_2=2$. The fourth equation is then automatically satisfied. Thus the unique positive solution is
\[
A_0=1,\qquad A_1=1,\qquad A_2=2,\qquad B_0=1,\qquad B_1=1.
\]
These heights give the trivial torus cover corresponding to $K_2$. Hence the case $K_2$ contributes no nontrivial rank $1.5$ orbit closure. Since the other surviving cases $J_1,J_3,J_5,$ and $K_3$ similarly give trivial torus covers, the computation in $\cH(2,0^3)$ produces no nontrivial examples.
\end{ex}

\bibliographystyle{amsalpha}
\bibliography{name}
\end{document}